\documentclass[11pt]{article}
\usepackage{mathtools}
\usepackage[T1]{fontenc}
\usepackage{amsfonts}
\usepackage{amsmath}
\usepackage{amssymb}
\usepackage{amsthm}
\usepackage{bbm}
\usepackage{bm}
\usepackage{mathrsfs}
\usepackage{color}
\usepackage{enumitem}
\usepackage[colorlinks=true, linkcolor=blue, citecolor=blue, urlcolor=blue]{hyperref}
\usepackage{tikz}
\usepackage[
  backend=biber,
  style=numeric,
  sorting=nyt,
  giveninits=true,
  maxbibnames=99,
  doi=true,
  isbn=false,
  url=false,
  eprint=true
]{biblatex}

\DeclareNameAlias{author}{family-given}
\DeclareNameAlias{editor}{family-given}
\DeclareFieldFormat[article]{title}{#1}
\DeclareFieldFormat[incollection]{title}{#1}
\DeclareFieldFormat[inproceedings]{title}{#1}
\DeclareFieldFormat[book]{title}{\mkbibemph{#1}}
\renewbibmacro{in:}{}

\AtEveryBibitem{%
  \clearfield{eprintclass}%
  \clearfield{primaryclass}%
}
\renewbibmacro*{eprint:arxiv}{%
  \printtext[eprint:arxiv]{%
    \printfield{eprinttype}%
    \setunit{\addcolon}%
    \printfield{eprint}}}

\DeclareMathOperator*{\argmax}{argmax}

\DeclareMathOperator{\supp}{supp}
\DeclareMathOperator{\dist}{dist}
\DeclareMathOperator{\Lip}{Lip}

\DeclareMathOperator*{\osc}{osc}
\DeclareMathOperator{\diam}{diam}

\newcommand{\RR}{\mathbb{R}}
\newcommand{\R}{\RR}

\newcommand{\eps}{\varepsilon}

\newcommand{\cL}{\mathcal{L}}

\newcommand{\X}{\mathcal{X}}
\newcommand{\Y}{\mathcal{Y}}

\newcommand{\Hil}{\mathcal{H}}
\newcommand{\1}{\mathbf{1}}
\newcommand{\DUAL}{{\rm QOT}^{\star}} 
\newcommand{\QOT}{{\rm QOT}}

\usepackage[top=30mm,bottom=30mm,left=31mm,right=31mm]{geometry}
\newcommand{\mykill}[1]{}

\DeclareRobustCommand{\MNg}[1]{{\color{gray}{#1}}}
\DeclareRobustCommand{\MNg}[1]{}
\usepackage[capitalize, noabbrev]{cleveref}
\crefname{equation}{}{} %

\theoremstyle{plain}
\newtheorem{theorem}{Theorem}[section]
\newtheorem{proposition}[theorem]{Proposition}

\newtheorem{corollary}[theorem]{Corollary}
\theoremstyle{definition}
\newtheorem{definition}[theorem]{Definition}
\newtheorem{remark}[theorem]{Remark}
\newtheorem{example}[theorem]{Example}
\newtheorem{assumption}[theorem]{Assumption}
\crefname{assumption}{Assumption}{Assumptions}
\Crefname{assumption}{Assumption}{Assumptions}
\theoremstyle{remark}
\AddToHook{env/theorem/begin}{\crefalias{section}{theorem}}
\AddToHook{env/proposition/begin}{\crefalias{theorem}{proposition}}
\AddToHook{env/lemma/begin}{\crefalias{theorem}{lemma}}
\AddToHook{env/corollary/begin}{\crefalias{theorem}{corollary}}
\AddToHook{env/definition/begin}{\crefalias{theorem}{definition}}
\AddToHook{env/remark/begin}{\crefalias{theorem}{remark}}
\AddToHook{env/example/begin}{\crefalias{theorem}{example}}
\AddToHook{env/assumption/begin}{\crefalias{theorem}{assumption}}
\crefname{theorem}{Theorem}{Theorems}
\crefname{proposition}{Proposition}{Propositions}
\crefname{lemma}{Lemma}{Lemmas}
\crefname{corollary}{Corollary}{Corollaries}
\crefname{definition}{Definition}{Definitions}
\crefname{remark}{Remark}{Remarks}
\crefname{example}{Example}{Examples}
\crefname{assumption}{Assumption}{Assumptions}

\renewenvironment{thebibliography}[1]{%
\begin{oldthebibliography}{#1}%
\setlength{\baselineskip}{.9em}
\linespread{1}
\small
\setlength{\parskip}{.40ex}%
\setlength{\itemsep}{.30em}%
}%
{%
\end{oldthebibliography}%
}
\usepackage{tocloft}

\begin{document}

\title{\vspace{-1em} Stability of Quadratically Regularized Optimal Transport}
\date{\today}
\author{  
  Alberto Gonz{\'a}lez-Sanz%
  \thanks{Department of Statistics, Columbia University, ag4855@columbia.edu.} \and  Marcel Nutz%
  \thanks{Departments of Mathematics and Statistics, Columbia University, mnutz@columbia.edu. Research supported by NSF Grants DMS-2106056, DMS-2407074.}%
  }

\maketitle

\vspace{-1.5em}
\begin{abstract}
Quadratically regularized optimal transport (QOT) is a sparse alternative to entropic optimal transport. We develop a quantitative stability theory for QOT under perturbations of the marginals, the transport cost function, and the regularization parameter. The centerpiece is an $L^\infty$-stability result for the dual potentials. Starting from an error bound in an $L^2$-space that varies with the marginals, we use a self-bound for the potentials to derive a local $L^\infty$-Lipschitz bound that is uniform over marginals. This bound also yields stability of the optimal coupling and of its support. In particular, we show that for the quadratic transport cost, the support of the optimal coupling is locally Lipschitz in Hausdorff distance under perturbations of the marginals. To the best of our knowledge, this is the first stability result for the optimal support in regularized optimal transport.
\end{abstract}

\vspace{1em}

{\small

}

\setcounter{tocdepth}{1}
\tableofcontents

\section{Introduction}

Optimal transport provides a variational framework for comparing probability
measures that has become ubiquitous across optimization, machine learning, statistics, and analysis. In applications, optimal transport is often regularized to reduce computational cost and improve statistical sample complexity. The best-studied regularization is entropic optimal transport (EOT), where the transport coupling is penalized by Kullback--Leibler
(KL) divergence; see, e.g.,
\cite{Cuturi.2013.Neurips,Peyre.Cuturi.2019.Book,Nutz.2021.LectureNotes}. EOT leads to Sinkhorn's algorithm and a smooth dual problem. However, the KL penalty forces the optimal coupling to have full support while unregularized optimal transport is typically sparse, and weak regularization can cause numerical instability
through exponentially large and small scalings~\cite{Schmitzer.19}. 
The most prominent alternative, first considered in \cite{Muzellec.2017.AAAI,blondel18quadratic,EssidSolomon.18}, is the quadratic regularization, or equivalently \(\chi^2\)-divergence. It produces sparse couplings, as observed in \cite{blondel18quadratic,EssidSolomon.18,Lorenz.2019,BayraktarEckstein.2025.BJ}
and recently analyzed theoretically in \cite{WieselXu.24,GonzalezSanzNutz2024.Scalar,Nutz.2024,gonzalezsanz2026sharplocalsparsityregularized}. It also replaces the exponential structure in the EOT dual by a quadratic, thereby avoiding exponential scaling. Recent work has shown that QOT enjoys parametric sample complexity~\cite{GonzalezSanzDelBarrioNutz.25} and linear convergence of several dual algorithms~\cite{GonzalezSanzNutzRiveros.25,GonzalezSanzNutzRiveros.26}, establishing QOT as a viable alternative to EOT. %

Given probability measures $P,Q$ on $\R^d$, a cost function $c(x,y)$, and a regularization parameter $\eps>0$, the QOT problem is %
\begin{equation*}\label{eq:qot-intro}
  \QOT_\eps(P,Q)
  :=
  \inf_{\pi\in\Pi(P,Q)}
  \int c(x,y)\,d\pi(x,y)
  +
  \frac{\eps}{2}
  \left\|
  \frac{d\pi}{d(P\otimes Q)}
  \right\|_{L^2(P\otimes Q)}^2,
\end{equation*}
where $\Pi(P,Q)$ denotes the set of couplings of $(P,Q)$. 
In practice, the most important example is the quadratic cost $c(x,y)=\frac12\|x-y\|^2$, but we shall develop the theory for a generic Lipschitz cost on a subset  $\X\times\Y\subset\R^d\times\R^d$. The dual variables of the QOT problem are the so-called potentials \(f,g\) solving the dual problem
\begin{equation*}\label{eq:qot-dual-intro}
  \sup_{(f,g)\in L^2(P)\times L^2(Q)}
  \int
  \left(
    f(x)+g(y)
    -\frac{1}{2\eps}\big(f(x)+g(y)-c(x,y)\big)_+^2
  \right)d(P\otimes Q)(x,y),
\end{equation*}
where \(t_+:=\max\{t,0\}\). Its first-order conditions of optimality are
\begin{equation}\label{eq:qot-foc-intro}
  \int \big(f(x)+g(y)-c(x,y)\big)_+\,dQ(y)=\eps,
  \qquad
  \int \big(f(x)+g(y)-c(x,y)\big)_+\,dP(x)=\eps
\end{equation}
for \(x\in\X\) and \(y\in\Y\), respectively. The potentials are unique up to translation under a mild connectedness assumption, so that their direct sum \(h=f\oplus g\) is uniquely determined. The primal optimizer $\pi$ is recovered through
\begin{equation}\label{eq:density-intro}
  \frac{d\pi}{d(P\otimes Q)}
  =
  \frac{1}{\eps}\sigma_+ \qquad \text{with the slack variable}\qquad \sigma:=h-c,
\end{equation}
and its support is
\begin{equation}\label{eq:opt-support-intro}
  \supp(\pi)
  =
  \overline{\{(x,y)\in\Omega_P\times\Omega_Q:\sigma(x,y)>0\}}.
\end{equation}
Thus \(h=f\oplus g\) determines the key objects of QOT, motivating us to view
\begin{equation}\label{eq:solution-map-intro}
  (P,Q,c,\eps)\mapsto h_{(P,Q,c,\eps)}
\end{equation}
as the solution map of our problem. 

The present paper provides a quantitative stability theory for this map. 
Our motivating question is the stability of the optimal support, about which, to the best of our knowledge, there are no prior results. In EOT, the optimal coupling always has full support (i.e., the same support as $P\otimes Q$), and hence support stability is not relevant. Since sparsity is the key difference between QOT and EOT, support stability is arguably the most interesting aspect among the stability properties. We answer this question by showing that the optimal support for quadratic cost is locally Lipschitz continuous (in Hausdorff distance) with respect to the marginals $(P,Q)$, with explicit constants detailed below. In view of~\eqref{eq:opt-support-intro}, our approach is to develop $L^\infty$-stability of $h=f\oplus g$ (and hence of $\sigma$), and to further ensure that $\sigma$ is not flat close to its zero-level set (which marks the boundary of the support).

Our analysis begins with a systematic study of the solution map~\eqref{eq:solution-map-intro}. For simplicity, we focus on perturbations of the marginals $(P,Q)$ in this introduction, and fix $c,\eps$. While our results in the main text cover joint variations of the data $(P,Q,c,\eps)$, varying $(c,\eps)$ is relatively straightforward because it does not change the function spaces. Our analysis can be contrasted with EOT, where the dual problem is strongly concave and the first-order condition for the potentials is the Schr\"odinger system 
\begin{equation}\label{eq:eot-schrodinger-intro}
\begin{aligned}
  f(x)
  &=
  -\eps\log\int
  \exp\left(\frac{g(y)-c(x,y)}{\eps}\right)\,dQ(y),\\
  g(y)
  &=
  -\eps\log\int
  \exp\left(\frac{f(x)-c(x,y)}{\eps}\right)\,dP(x).
\end{aligned}
\end{equation}
As shown by \cite{CarlierLaborde.20}, stability results can be derived by applying the implicit function theorem to this system. For QOT, the dual problem is not strongly concave and this approach fails. However, the recent work \cite{GonzalezSanzNutzRiveros.26} showed that a local PL inequality/error bound (see \cref{thm:dual-error-bound} below) holds in $L^2(P\otimes Q)$. It is well known (e.g., \cite[Section~4.4]{BonnansShapiro.00}) that a PL inequality implies a stability property. However, a direct application would only lead to a limit result: to work with a fixed $L^2$ space, one would have to restrict to marginals that are absolutely continuous with respect to fixed reference measures $P_0,Q_0$. By contrast, \cref{thm:l2-stability} and \cref{cor:symmetric-l2-stability} below provide local $L^2$-Lipschitz estimates where the input data are compared in the 1-Wasserstein metric. (Here ``local'' means that the marginals under consideration must have $W_1$-distance below an explicit threshold.)

$L^2$-stability serves as a stepping stone for the $L^\infty$-stability provided in \cref{thm:linfty-stability}  and \cref{cor:solution-map-lipschitz}. The upgrade from $L^2$ to $L^\infty$ is based on a self-bound for the potentials. Such a self-bound is  straightforward for EOT, where the Schr\"odinger system \eqref{eq:eot-schrodinger-intro} directly expresses $f,g$ as integrals of themselves. For QOT, the self-bound is nontrivial, but can be shown by a monotonicity argument and~\eqref{eq:qot-foc-intro} once the potentials are sufficiently close in $L^2$ (see Step~3 in the proof of \cref{thm:linfty-stability}). 

Stability of the primal objects follows from the dual estimates. Indeed, stability of the optimal density follows directly via~\eqref{eq:density-intro}. This density is with respect to a reference $P\otimes Q$ which is itself varying, but nevertheless allows us to conclude local Lipschitz continuity of the optimal coupling in both total variation and 1-Wasserstein distance. The primal stability results are stated in \cref{thm:primal-stability} and \cref{rem:primal-density-linfty-stability}.

We finally return to the optimal support. Potential stability alone cannot
control supports: when \(c\equiv0\), the optimal coupling is \(P\otimes Q\) and
\(h\equiv\eps\), so the potentials do not register changes in the marginal
supports. Our result therefore takes the Hausdorff displacement of the marginal supports as input, in addition to the quantities controlling the potentials. Under a nondegeneracy condition, \cref{thm:support-stability} then establishes local Lipschitz continuity of the optimal support in Hausdorff distance. The nondegeneracy condition, detailed in~\eqref{eq:exterior-nondegeneracy}, states that the slack $\sigma=h-c$ detaches linearly outside the optimal support. Intuitively, this means that the gradient of $\sigma$ is bounded away from zero at the boundary of the support (which is itself the zero-level set of $\sigma$). Whether this holds depends on the geometry of the cost $c$. For the quadratic
cost \(c(x,y)=\frac12\|x-y\|^2\), the slack $\sigma$ is concave in each coordinate, and we show  in \cref{pr:concavity-exterior-nondegeneracy} that this implies the nondegeneracy condition with a slope of at least $\eps/\diam(\supp P)$. As a result, the optimal support is indeed stable for the canonical cost (\cref{cor:quadratic-support-stability}), answering our initial question. On the other hand, the nondegeneracy condition is essential for support stability and can fail for certain costs: in \cref{ex:failure-support-stability}, we construct a setting where the marginals converge, the marginal supports do not move, and
the potentials converge uniformly, but the optimal supports fail to converge. 

\paragraph{Related literature.}
Classical optimal transport has a substantial stability theory, going back to \cite{cuesta1997optimal}, which proves weak convergence of optimal couplings under weak convergence of the marginal distributions. Quantitative estimates were given by \cite{Gigli.2011.PEMS}, which
studies the optimal map from a fixed source to targets moving along a Wasserstein curve. More recently, \cite{DelalandeMerigot.2023.DMJ} proves bi-H\"older stability of Brenier maps
under Wasserstein perturbations of the target with dimension-free exponents. See, e.g., \cite{GonzalezSanzSheng.2024.ArXiv,LetrouitMerigot.2024.ArXiv,Merigot.26} for more extensive references. 

\emph{Quadratic and general divergence regularization.}
The closest stability result for non-entropic regularization is
\cite{BayraktarEckstein.2025.BJ}, which studies divergence-regularized optimal transport for a broad class of divergences, including the
\(\chi^2\)-divergence underlying QOT. The authors prove $\frac{1}{2p}$-H\"older stability of
the optimal couplings in $p$-Wasserstein distance under perturbations of the marginals, and also derive sample-complexity bounds for empirical marginals. Previously, \cite{EcksteinNutz.22} showed Lipschitz continuity of the optimal value. Both works are based on the shadow technique and the data processing inequality, which do not control the potentials or the supports. More recently, \cite{GonzalezSanzDelBarrioNutz.25} proves parametric sample complexity and central limit theorems for QOT, but does not address deterministic stability or the optimal supports. In a different direction, \cite{Eckstein.Nutz.2023} studies convergence for vanishing regularization parameter $\eps\to0$, focusing on the optimal value rather than the optimizers. 

Several works analyze the structure of QOT and the support of its optimal coupling. Duality is developed in
\cite{Lorenz.2019} and \cite{Nutz.2024}; the latter also establishes a qualitative sparsity result and uniqueness of the potentials. For quadratic transport cost, sparsity is studied quantitatively in
\cite{WieselXu.24,GonzalezSanzNutz2024.Scalar,gonzalezsanz2026sharplocalsparsityregularized}. 
These works describe the geometry of the support for small \(\eps\) and
fixed marginals; they do not investigate stability. On the algorithmic side, \cite{GonzalezSanzNutzRiveros.25}
proves linear convergence of gradient descent for the QOT dual. Its proof technique is foundational for the subsequent work
\cite{GonzalezSanzNutzRiveros.26} proving a local error bound and Polyak--\L{}ojasiewicz inequality for the QOT dual in $L^2(P\otimes Q)$, where the marginals $P,Q$ are fixed. The present paper uses the error bound of \cite{GonzalezSanzNutzRiveros.26} as a starting point, but the main
developments here are different: the underlying Hilbert space \(L^2(P\otimes Q)\) moves with the marginals $P,Q$, and support stability in the Hausdorff sense requires $L^\infty$-control of the potentials.

\emph{Entropic regularization.}
For entropic optimal transport, stability is much better developed. The differential
approach of \cite{CarlierLaborde.20} proves well-posedness and smooth dependence
of the Schr\"odinger system in an \(L^\infty\) setting via the implicit function theorem. Refining this approach, \cite{CarlierChizatLaborde.24} show that the potentials are \(C^k\)-Lipschitz from the \(L^2\)-Wasserstein
space of marginals when the transport cost is \(C^{k+1}\).  In a more general setting, \cite{GhosalNutzBernton.22} use a geometric approach based on cyclical invariance to prove qualitative stability of the primal EOT optimizers for weak convergence with respect to perturbations of the marginals and cost. Analyzing Sinkhorn's algorithm in a bounded setting, \cite{DeligiannidisDeBortoliDoucet.24} show that its iterates are $L^\infty$-Lipschitz uniformly over the iteration
number, and hence that the potentials are $L^\infty$-Lipschitz, as functions of the marginals in 1-Wasserstein distance. In an unbounded setting, \cite{EcksteinNutz.22} establish Lipschitz continuity of the entropic value and H\"older continuity of the optimal coupling in Wasserstein distance using the shadow construction. 
Stability of entropic Brenier maps is treated in
\cite{DivolNilesWeedPooladian.24}, which obtains sharp Lipschitz constants under target perturbations and improves the dependence on the regularization parameter in the semi-discrete regime. Recent contributions include \cite{YaoKimSchiebinger.26}, which proves quantitative stability for many-marginal Schr\"odinger bridges. For the reasons explained above, the techniques used in EOT do not directly extend to QOT. Moreover, the question of support stability does not arise in EOT.

\paragraph{Organization.}
\Cref{sec:setting-background} details the setting and collects basic facts about QOT. \Cref{sec:l2-stability} establishes the \(L^2\)-stability of the potentials, whereas \(L^\infty\)-stability is shown in \cref{sec:Linfty-stability}. \Cref{sec:primal-stability}
derives stability of the optimal densities and couplings. \Cref{sec:support-stability}
concludes with the Hausdorff stability of the optimal support. All proofs are gathered in Appendix~\ref{sec:proofs}.

\section{Setting and background}\label{sec:setting-background}

 We fix two subsets
\(\X,\Y\subset\R^d\) and denote by \(\Omega_\nu\) the topological
support of a measure \(\nu\). All first
marginal probability measures considered in this paper are assumed to have
finite second moments and support contained in \(\X\), and all second
marginal probability measures are assumed to have finite second moments and
support contained in \(\Y\). The main purpose of the sets $\X,\Y$ is to localize the problem so
that, e.g., the quadratic cost becomes Lipschitz. If $c$ is Lipschitz on $\R^d\times\R^d$, one may simply take \(\X=\Y=\R^d\). Throughout, we use the Euclidean norms on $\R^d$ and $\R^d\times\R^d$ unless
otherwise stated.%

\begin{assumption}[Cost] \label{ass:cost}
The cost \(c:\X\times \Y\to \R\) is Lipschitz with constant \(L>0\).
\end{assumption}

The quadratically regularized optimal transport (QOT) problem with
regularization parameter \(\eps>0\) is
\begin{equation}\label{eq:qot-primal}
  \QOT_\eps(P,Q):=  \inf_{\pi\in \Pi(P,Q)} \int c(x,y)\, d\pi(x,y) +\frac{\eps }{2}\left\| \frac{d \pi}{ d(P\otimes Q)}\right\|^2_{L^{2}(P\otimes Q)}
\end{equation}
with the convention that the last term is \(+\infty\) if
\(\pi \not\ll P\otimes Q\). The dual problem is
\begin{equation}\label{eq:qot-dual}
     \DUAL_\eps(P,Q)=\sup_{(f,g)\in L^2(P)\times L^2(Q)} \Phi(f\oplus g),
\end{equation}
where \(f\oplus g(x,y):=f(x)+g(y)\) denotes the direct sum and the dual
objective function \(\Phi\) is
\begin{align}\label{eq:dual-objective}
     \Phi(f \oplus g) := \int \left(f(x) + g(y) -\frac{1}{2\eps} \left(f(x)+ g(y)- c(x,y)\right)_+^{2}\right) d(  P \otimes   Q )(x,y).
\end{align}

The next proposition summarizes several standard properties. For a function
\(w\) defined on a set \(A\), we write
\(\|w\|_{L^\infty(A)}:=\sup_{z\in A}|w(z)|\).

\begin{proposition}\label{prop:qot-basic-properties}
    Let \(P,Q\) be probability measures on \(\R^d\) with finite second moments and let \(c\) satisfy \cref{ass:cost}. Then:
    \begin{enumerate}
    \item
    The strong duality
      \(\QOT_\eps(P,Q)=\DUAL_\eps(P,Q)\)
    holds.
    \item
    The dual problem~\eqref{eq:qot-dual} admits an optimizer \((f,g)\in L^2(P)\times L^2(Q)\). If \(\Omega_P\) is connected,
    the functions \((f,g)\) are unique in \(L^1(P)\times L^1(Q)\) up to translation, in the sense that fixing one optimizer \((f,g)\), the set of all optimizers is \(\{(f+a,g-a): a\in\R\}\).

    \item We can choose versions \(f:\X\to\R\) and \(g:\Y\to\R\) that are \(L\)-Lipschitz and satisfy the first-order condition
    \begin{equation}
        \label{eq:potential-foc}
        \begin{cases}
             \int (f(x)+ g(y)-c(x,y))_+\, dQ(y)=\eps \quad \text{for all }x\in\X,\\
             \int (f(x)+ g(y)-c(x,y))_+\, dP(x)=\eps \quad \text{for all }y\in\Y.
        \end{cases}
    \end{equation}
    In the following, we always choose such versions, and call them
    potentials. If \(\Omega_P\) is connected, they are again unique up to
    translation.

    \item If \(c\) is bounded on \(\X\times\Y\), then so are
    \((f,g)\). More precisely, their oscillations are bounded as
    \(\osc_{\X}(f),\osc_{\Y}(g)\le \osc_{\X\times\Y}(c)\leq 2\|c\|_{L^\infty(\X\times\Y)}\), and
    \begin{equation}\label{eq:potential-sum-bounds}
        -5\|c\|_{L^\infty(\X\times\Y)} + \eps \le
        f(x)+g(y)
        \le
        5\|c\|_{L^\infty(\X\times\Y)}+\eps, \qquad (x,y)\in\X\times\Y.
    \end{equation}
    \item
    The primal problem~\eqref{eq:qot-primal} has a unique solution
    \(\pi\in\Pi(P,Q)\). It is related to the dual by
    \begin{equation}\label{eq:optimal-density}
    \frac{d\pi}{d(P \otimes Q)}(x,y) = \frac{1}{\eps} \big( f(x) + g(y) - c(x,y) \big)_+ \qquad P \otimes Q\text{-a.s.}
    \end{equation}
    and its support is 
    \begin{equation}
    \label{eq:optimal-support}
    \supp(\pi)
    =
    \overline{\{(x,y)\in \Omega_P\times\Omega_Q:f(x)+g(y)-c(x,y)>0\}}.
    \end{equation}
\end{enumerate}
\end{proposition}

Let \(\cL_{d}\) denote the \(d\)-dimensional Lebesgue measure and \(B_r(x)\)
the open ball of center \(x\) and radius \(r>0\). The following additional
conditions on the reference marginals \(P,Q\) are used in the error bound (\cref{thm:dual-error-bound}) below.

\begin{assumption}[Reference marginals]\label{ass:reference-marginals}
(a) The first marginal \(P\) has convex, compact support \(\Omega_P\) and
admits a density \(\rho:=dP/d\cL_{d}\) that is bounded away from zero and
infinity on \(\Omega_P\), i.e., there exist constants
\(0<\lambda_P\leq\Lambda_P<\infty\) such that
\[
 \lambda_P\leq \rho(x) \leq \Lambda_P \quad \text{for all } x\in \Omega_P.
\]
(b) The second marginal \(Q\) has compact support \(\Omega_Q\).
\end{assumption}

\Cref{ass:reference-marginals} implies the following bounds on ball measures.

\begin{remark}[Lower bounds for ball measures]\label{rem:ball-measure-lower-bounds}
(a) Since \(\Omega_P\) is compact and convex, it satisfies a uniform interior
cone condition: for every \(x\in \Omega_P\), there exists a convex cone
\(\mathcal{C}_x\) with vertex \(x\), angle \(\theta>0\) and height \(h>0\)
such that \(\mathcal{C}_x\subset \Omega_P\), where \(\theta\) and \(h\) are
independent of \(x\) (e.g., \cite[p.\,12]{Grisvard.85}). Hence,
\cref{ass:reference-marginals} implies that there exists
\(\delta_P\in (0,1]\), depending only on \(\Omega_P\) and \(\lambda_P\),
such that
\[
P(B_r(x)) \geq  \delta_P\, \min(r^d, 1) \quad \text{for all }r>0 \text{ and }x\in \Omega_P.
\]
(b) Since \(\Omega_Q\) is compact, lower semicontinuity of
\(y\mapsto Q(B_r(y))\) implies
\[
\inf_{y\in \Omega_Q }Q(B_r(y))>0 \quad \text{for all }r>0.
\]
\end{remark}

Next, consider the linear subspace
\[
  \Hil = \big\{u\oplus v: (u,v)\in L^{2}(P)\times L^{2}(Q)\big\} \subset  L^{2}(P\otimes Q),
\]
which is closed by the argument in \cite[p.\,370]{RuschendorfThomsen.93}.
Hence, \(\Hil\) is a Hilbert space with the induced inner product
\(\langle \cdot,\cdot\rangle_{\Hil} = \langle \cdot,\cdot\rangle_{L^{2}(P\otimes Q)}\).
In the following, we consider the dual objective \(\Phi\) of~\eqref{eq:dual-objective}
as an operator
\[
\Phi: \Hil \to \RR
\]
and denote by \({\rm D}\Phi(h)\) its \(\Hil\)-gradient. For any
\(h,w\in\Hil\), the action of \({\rm D}\Phi(h)\) on \(w\) is
\begin{equation}\label{eq:dual-gradient-action}
    \langle {\rm D}\Phi(h),w\rangle_{L^{2}(P\otimes Q)}
    =
    \int w\,R_h\,d(P\otimes Q),
\end{equation}
where
\[
R_h:=1-\frac{1}{\eps}(h-c)_+ \in L^2(P\otimes Q).
\]
In particular,
\begin{equation}\label{eq:dual-gradient-norm}
    \|{\rm D}\Phi(h)\|_{L^{2}(P \otimes Q)} = \sup_{w\in\Hil:\, \|w\|_{L^{2}(P\otimes Q)}\leq1} \int w\,R_h\,d(P\otimes Q).
\end{equation}
This norm governs the right-hand side in the following local error bound for
the dual problem, established in
\cite[Theorem~3.1]{GonzalezSanzNutzRiveros.26}.

\begin{theorem}[Error bound]\label{thm:dual-error-bound}
    Let \((P,Q)\) satisfy \cref{ass:reference-marginals}, let \(c\) satisfy
\cref{ass:cost}, let \(\eps>0\), and let \((f,g)\) be the
associated potentials. Setting \(h=f\oplus g\in \Hil\), the error bound
    \begin{equation*}
    \|\tilde{h}-h\|_{L^2(P\otimes Q)}
    \leq
    \gamma_\eps
    \max\left(\|\tilde{h}-h\|_{L^\infty(P\otimes Q)},\eps\right)
    \|{\rm D}\Phi(\tilde{h})\|_{L^2(P\otimes Q)}
    \end{equation*}
    holds for all \(\tilde{h}\in \Hil\), where
    \[
    \gamma_\eps= 16 \left(\delta_P^{-1}\max\left( \frac{8L}{\eps},1\right)^d\right) \frac{\Lambda_P^2}{\lambda_P^2}
    \frac{\bigl(\lceil 8L\,{\rm diam}(\Omega_P)/\eps \rceil\bigr)^{d+2}}{\inf_{y\in\Omega_Q} Q\bigl(B_{\frac{\eps}{8L}}(y)\bigr)}.
    \]
\end{theorem}

\section{Stability of the potentials in \(L^2\)}
\label{sec:l2-stability}

Our first result reformulates the local error bound into an $L^2$-estimate for nearby data. The nontrivial part is (ii), establishing that the locality condition---which is in the stronger $L^\infty$-norm---is automatically satisfied when the data are close enough. 

\begin{theorem}[$L^2$-Lipschitz stability]
\label{thm:l2-stability}
Let \(\eps,\eps'>0\), let \(c,c'\) satisfy
\cref{ass:cost}, let \((P,Q)\) satisfy
\cref{ass:reference-marginals}, and let \((P',Q')\) be any marginals with
finite second moments. Let \(f,g\) be the potentials (unique up to
translation) associated with \((P,Q,c,\eps)\), and let \(f',g'\) be (any)
potentials associated with \((P',Q',c',\eps')\). Set \(h:=f\oplus g\) and
\(h':=f'\oplus g'\), and
\begin{align*}
\Delta&:=2L\,\Big(W_1(P,P')^2+W_1(Q,Q')^2\Big)^{1/2}+\|c-c'\|_{L^2(P\otimes Q)}+|\eps-\eps'|.
\end{align*}
Then, the following hold.
\begin{enumerate}
\item[(i)]
Local error-bound: We have
\begin{equation}\label{eq:l2-local-error-bound}
\|h-h'\|_{L^2(P\otimes Q)}
\le
\frac{\gamma_\eps}{\eps}
\max\big(\|h-h'\|_{L^\infty(\Omega_P\times\Omega_Q)},\eps\big)
\,\Delta.
\end{equation}

\item[(ii)]
Modulus of continuity into \(L^\infty\): For \(\delta>0\), define
\begin{equation}\label{eq:vartheta-delta}
\vartheta_\delta
:=
\delta_P \min\Big\{\Big(\frac{\delta}{8L}\Big)^d,1\Big\}
\inf_{y\in\Omega_Q} Q\Big(B_{\delta/(8L)}(y)\Big).
\end{equation}
Then
\begin{equation}\label{eq:l2-linfty-modulus}
\Delta
<
\frac{\min(\delta,\eps)\sqrt{\vartheta_\delta}}{2\gamma_\eps}
\qquad \implies \qquad
\|h-h'\|_{L^\infty(\Omega_P\times\Omega_Q)}<\delta.
\end{equation}
\item[(iii)] Local Lipschitz continuity in \(L^2\): We have
\begin{equation*}
\|h-h'\|_{L^2(P\otimes Q)}
\le \gamma_\eps\,\Delta \qquad \text{whenever }\Delta<\frac{\eps\sqrt{\vartheta_\eps}}{2\gamma_\eps}.
\end{equation*}
\end{enumerate}
\end{theorem}

The estimate in \cref{thm:l2-stability} is asymmetric---both the norm and
the constants are tied to the reference datum \((P,Q,c,\eps)\). The next
definition introduces a class on which the relevant constants are uniform,
yielding a symmetric corollary.

\begin{definition}[Class \(\mathfrak D\) of data]
\label{def:admissible-class}
Fix constants
\[
\underline\eps,D,L\in(0,\infty),\qquad
0<\underline\lambda\le \overline\Lambda<\infty,\qquad
0<\underline\delta,\underline q\le 1.
\]
Let \(\mathfrak D=\mathfrak D(\underline\eps,D,L,\underline\lambda,\overline\Lambda,\underline\delta,\underline q)\) be the class of all quadruples \((P,Q,c,\eps)\) such that:
\begin{enumerate}
\item[(a)] \(\eps\ge \underline\eps\);
\item[(b)] \(c:\X\times\Y\to\R\) is \(L\)-Lipschitz;
\item[(c)] \(\Omega_P\) is compact and convex, and \(\diam(\Omega_P)\le D\);
\item[(d)] \(P\) admits a density \(\rho=dP/d\cL_d\) satisfying \(\underline\lambda\le \rho(x)\le \overline\Lambda\) for all \(x\in\Omega_P\);
\item[(e)] \(P\) satisfies \(P(B_r(x))\ge \underline\delta\,\min(r^d,1)\) for all \(x\in\Omega_P\), \(r>0\);
\item[(f)] \(\Omega_Q\) is compact and \(\inf_{y\in\Omega_Q} Q\bigl(B_{\underline\eps/(8L)}(y)\bigr)\ge \underline q\).
\end{enumerate}
\end{definition}

The preceding theorem becomes symmetric on this class, after applying it with each datum as the reference.

\begin{corollary}[Uniform \(L^2\)-stability]
\label{cor:symmetric-l2-stability}
Let \(\mathfrak D\) be as in \cref{def:admissible-class}, and define
\begin{align}
\overline\gamma
&:=
16\left(
\underline\delta^{-1}\max\Big(\frac{8L}{\underline\eps},1\Big)^d
\right)
\frac{\overline\Lambda^2}{\underline\lambda^2}
\frac{\bigl(\lceil 8L\,D/\underline\eps\rceil\bigr)^{d+2}}
{\underline q}
\nonumber\\
\underline\vartheta
&:=
\underline\delta
\min\Big\{\Big(\frac{\underline\eps}{8L}\Big)^d,1\Big\}\underline q
\label{eq:uniform-l2-constants}
\nonumber\\
\overline\eta
&:=
\frac{\underline\eps\sqrt{\underline\vartheta}}{2\overline\gamma}.
\nonumber
\end{align}
For \((P,Q,c,\eps)\in\mathfrak D\), let
$
h_{(P,Q,c,\eps)}:=f\oplus g,
$
where \((f,g)\) are the associated potentials. Given
\[
(P,Q,c,\eps),\ (P',Q',c',\eps')\in\mathfrak D,
\]
abbreviate
$h:=h_{(P,Q,c,\eps)}$ and $h':=h_{(P',Q',c',\eps')}$, and set 
\begin{align*}
\Delta_W
&:=
\Big(W_1(P,P')^2+W_1(Q,Q')^2\Big)^{1/2}
\\
\Delta
&:=
2L\,\Delta_W+\|c-c'\|_{L^2(P\otimes Q)}+|\eps-\eps'|
\\
\Delta'
&:=
2L\,\Delta_W+\|c-c'\|_{L^2(P'\otimes Q')}+|\eps-\eps'|
\\
\bar\mu
&:=
\frac{P\otimes Q+P'\otimes Q'}{2}
\\
\bar\Delta
&:=
2L\,\Delta_W+\|c-c'\|_{L^2(\bar\mu)}+|\eps-\eps'|.
\end{align*}
If
\begin{equation*}
\label{eq:symmetric-l2-smallness}
\max\{\Delta,\Delta'\}<\overline\eta,
\end{equation*}
then
\begin{align}
\label{eq:symmetric-l2-bound-unprimed}
\|h-h'\|_{L^2(P\otimes Q)}
&\le \overline\gamma\,\Delta \\
\label{eq:symmetric-l2-bound-primed}
\|h-h'\|_{L^2(P'\otimes Q')}
&\le \overline\gamma\,\Delta' \\
\label{eq:symmetric-l2-bound-mixture}
\|h-h'\|_{L^2(\bar\mu)}
&\le \overline\gamma\,\bar\Delta.
\end{align}
\end{corollary}

\section{Stability of the potentials in \(L^\infty\)}
\label{sec:Linfty-stability}

Our next result upgrades the \(L^2(P\otimes Q)\)-estimate of \cref{thm:l2-stability} to an \(L^\infty(\X\times\Y)\)-estimate. This is important in two respects. First, it removes the dependence on the marginals in the definition of the norm measuring the error, thus allowing for a statement that is uniform across a class of marginals. Second, \(L^\infty\)-control will be crucial to control the optimal support in Hausdorff sense in \cref{sec:support-stability}.

\begin{theorem}[\(L^\infty\)-Lipschitz stability]
\label{thm:linfty-stability}
Let $\eps,\eps'>0$, let $c,c'$ satisfy
\cref{ass:cost}, let \((P,Q)\) satisfy \cref{ass:reference-marginals}, and let $(P',Q')$ be any marginals with finite second moments. Let $f,g$ be the potentials (unique up to translation) associated with \((P,Q,c,\eps)\), and let $f',g'$ be (any) potentials associated with \((P',Q',c',\eps')\). Set $h:=f\oplus g$ and  $h':=f'\oplus g'$, let \(\gamma_\eps\) be as in \cref{thm:dual-error-bound}, and let \(\vartheta_\eps\) be as in~\eqref{eq:vartheta-delta}.
Define
\begin{align}
\Delta_* 
&:= d_{\mathfrak D}\big((P,Q,c,\eps),(P',Q',c',\eps')\big) \nonumber\\
&:=2L\,\Big(W_1(P,P')^2+W_1(Q,Q')^2\Big)^{1/2}+\|c-c'\|_{L^\infty(\X\times\Y)}+|\eps-\eps'|
\label{eq:distance-D}\\
\widehat q_\eps
&:=\inf_{y\in\Omega_Q}Q\Big(B_{\eps/(4L)}(y)\Big)
\nonumber\\
\widehat\kappa_\eps
&:=\delta_P\min\Big\{\Big(\frac{\eps}{4L}\Big)^d,1\Big\}
\nonumber\\
\widehat\eta_\eps
&:=
\min\left\{
\frac{\eps\sqrt{\vartheta_\eps}}{2\gamma_\eps},\,
\frac{\eps\,\widehat q_\eps}{2(1+\gamma_\eps)},\,
\frac{\eps\,\widehat\kappa_\eps}{2(1+\gamma_\eps)}
\right\}.
\label{eq:linfty-local-threshold}
\end{align}
Then \(\widehat q_\eps,\widehat\kappa_\eps,\widehat\eta_\eps>0\). If
\begin{equation}\label{eq:linfty-smallness}
\Delta_*<\widehat\eta_\eps,
\end{equation}
then, writing $a:=\int (g'-g)\,dQ$, we have
\begin{align}
\|f-f'-a\|_{L^\infty(\X)}
&\le
(1+\gamma_\eps)\widehat q_\eps^{-1}\,\Delta_*,
\label{eq:linfty-f-bound}
\\
\|g-g'+a\|_{L^\infty(\Y)}
&\le
(1+\gamma_\eps)\widehat\kappa_\eps^{-1}\,\Delta_*,
\label{eq:linfty-g-bound}
\end{align}
and consequently
\begin{equation}\label{eq:linfty-h-bound}
\|h-h'\|_{L^\infty(\X\times \Y)}
\le
(1+\gamma_\eps)\bigl(\widehat q_\eps^{-1}+\widehat\kappa_\eps^{-1}\bigr)\Delta_*.
\end{equation}
\end{theorem}

By making the constants uniform over a class \(\mathfrak D\) of data, we immediately obtain a uniform version of \cref{thm:linfty-stability}.

\begin{corollary}[Uniform \(L^\infty\)-stability]
\label{cor:solution-map-lipschitz}
Let \(\mathfrak D=\mathfrak D(\underline\eps,D,L,\underline\lambda,\overline\Lambda,\underline\delta,\underline q)\) be as in \cref{def:admissible-class}. For \((P,Q,c,\eps)\in\mathfrak D\), let
$h_{(P,Q,c,\eps)}=f\oplus g$, where $(f,g)$ are the associated potentials. Then the map\footnote{If $c$ is bounded on $\X\times\Y$, then this is a map into $L^\infty(\X\times\Y)$. In general, the difference $h_{(P,Q,c,\eps)}-h_{(P',Q',c',\eps')}$ in~\eqref{eq:uniform-linfty-lipschitz} is in $L^\infty(\X\times\Y)$, even though the two terms need not be bounded individually.}
\[
\mathfrak D\ni (P,Q,c,\eps)\longmapsto h_{(P,Q,c,\eps)}
\]
is uniformly locally Lipschitz with respect to
\(d_{\mathfrak D}\), defined in~\eqref{eq:distance-D}. More precisely, let
\begin{align}
\overline\gamma
&:=
16\left(
\underline\delta^{-1}\max\Big(\frac{8L}{\underline\eps},1\Big)^d
\right)
\frac{\overline\Lambda^2}{\underline\lambda^2}
\frac{\bigl(\lceil 8L\,D/\underline\eps\rceil\bigr)^{d+2}}
{\underline q}
\nonumber\\
\underline\vartheta
&:=
\underline\delta
\min\Big\{\Big(\frac{\underline\eps}{8L}\Big)^d,1\Big\}\underline q
\nonumber\\
\underline{\widehat\kappa}
&:=
\underline\delta
\min\Big\{\Big(\frac{\underline\eps}{4L}\Big)^d,1\Big\}
\nonumber\\
\overline\eta_*
&:=
\min\left\{
\frac{\underline\eps\sqrt{\underline\vartheta}}{2\overline\gamma},\,
\frac{\underline\eps\,\underline q}{2(1+\overline\gamma)},\,
\frac{\underline\eps\,\underline{\widehat\kappa}}{2(1+\overline\gamma)}
\right\},
\nonumber\\
\overline C
&:=
(1+\overline\gamma)\bigl(\underline q^{-1}+\underline{\widehat\kappa}^{-1}\bigr)
\label{eq:overline-C}.
\end{align}
Then, for any pair \((P,Q,c,\eps),(P',Q',c',\eps')\in\mathfrak D\) satisfying
\[
d_{\mathfrak D}\big((P,Q,c,\eps),(P',Q',c',\eps')\big)<\overline\eta_*
\]
we have
\begin{equation}\label{eq:uniform-linfty-lipschitz}
\|h_{(P,Q,c,\eps)}-h_{(P',Q',c',\eps')}\|_{L^\infty(\X\times\Y)}
\le
\overline C\,
d_{\mathfrak D}\big((P,Q,c,\eps),(P',Q',c',\eps')\big).
\end{equation}
\end{corollary}

\section{Stability of the optimal coupling}
\label{sec:primal-stability}

We now turn to the stability of the primal optimizers. Throughout this section, we work on the admissible class \(\mathfrak D=\mathfrak D(\underline\eps,D,L,\underline\lambda,\overline\Lambda,\underline\delta,\underline q)\) from \cref{def:admissible-class}. Given probability measures \(\mu,\nu\) on a measurable space, we write
\[
\|\mu-\nu\|_{\mathrm{TV}}
:=
\sup_{A} |\mu(A)-\nu(A)|
\]
for the total variation distance. The next theorem provides stability estimates for the optimal coupling as well as its density with respect to the product $P\otimes Q$ of the marginals (where $P\otimes Q$ is itself varying).

\begin{theorem}[Stability of primal optimizers]
\label{thm:primal-stability}
Let \(\mathfrak  D=\mathfrak D(\underline\eps,D,L,\underline\lambda,\overline\Lambda,\underline\delta,\underline q)\) be as in \cref{def:admissible-class}. Let \((P,Q,c,\eps),\ (P',Q',c',\eps')\in \mathfrak D\), let \(h,h'\) be the corresponding direct sums of potentials, let \(\pi\in\Pi(P,Q),\pi'\in\Pi(P',Q')\) be the corresponding optimal couplings, and denote their densities as 
\[
\zeta:=\frac{1}{\eps}(h-c)_+,
\qquad
\zeta':=\frac{1}{\eps'}(h'-c')_+.
\]
Set
\[
\mu:=P\otimes Q,
\qquad
\mu':=P'\otimes Q',
\qquad
\bar\mu:=\frac{\mu+\mu'}{2},
\]
\[
\Delta_W
:=
\Big(W_1(P,P')^2+W_1(Q,Q')^2\Big)^{1/2},
\]
\[
\Delta_{\mathrm{TV}}
:=
\|P-P'\|_{\mathrm{TV}}+\|Q-Q'\|_{\mathrm{TV}},
\]
and
\begin{align*}
\Delta
&:=
2L\,\Delta_W+\|c-c'\|_{L^2(\mu)}+|\eps-\eps'|,
\\
\Delta'
&:=
2L\,\Delta_W+\|c-c'\|_{L^2(\mu')}+|\eps-\eps'|,
\\
\bar\Delta
&:=
2L\,\Delta_W+\|c-c'\|_{L^2(\bar\mu)}+|\eps-\eps'|.
\end{align*}
Let \(\overline\gamma\) and \(\overline\eta\) be the constants from \cref{cor:symmetric-l2-stability} and define
\[
\Gamma
:=
(\Omega_P\cup\Omega_{P'})\times(\Omega_Q\cup\Omega_{Q'}),
\qquad
\Theta
:=
(\Omega_P\times\Omega_Q)\cup(\Omega_{P'}\times\Omega_{Q'}),
\]
\[
D_*:=\sup_{z,z'\in\Theta}\|z-z'\|, \qquad A:=1+\frac{6\|c\|_{L^\infty(\Gamma)}}{\eps},
\qquad
A':=1+\frac{6\|c'\|_{L^\infty(\Gamma)}}{\eps'},
\]
and
\begin{equation*}
\label{eq:primal-density-error-constant}
\hat\Delta
:=
\min\left\{
\frac{\overline\gamma\,\bar\Delta+\|c-c'\|_{L^2(\bar\mu)}+A'|\eps-\eps'|}{\eps},
\frac{\overline\gamma\,\bar\Delta+\|c-c'\|_{L^2(\bar\mu)}+A|\eps-\eps'|}{\eps'}
\right\}.
\end{equation*}
Assume that
\[
\max\{\Delta,\Delta'\}<\overline\eta,
\]
then the following hold.
\begin{enumerate}
\item[(i)] $L^2$-stability of the density:
\begin{equation}
\label{eq:primal-density-l2-bound}
\|\zeta-\zeta'\|_{L^2(\bar\mu)}\le \hat\Delta.
\end{equation}

\item[(ii)] TV-stability of the optimal coupling:
\begin{equation}
\label{eq:primal-tv-bound}
\|\pi-\pi'\|_{\mathrm{TV}}
\le
\frac{\hat\Delta}{2}
+
\frac{A+A'}{2}\,\Delta_{\mathrm{TV}}.
\end{equation}

\item[(iii)] $W_1$-stability of the optimal coupling:
\begin{equation}
\label{eq:primal-w1-bound}
W_1(\pi,\pi')
\le
\frac{D_*\,\hat\Delta}{2}
+
\sqrt2\left[
\frac{A+A'}{2}
+
\frac{(\sqrt2+1)L\,D_*}{4}\Big(\frac{1}{\eps}+\frac{1}{\eps'}\Big)
\right]\Delta_W.
\end{equation}
\end{enumerate}
\end{theorem}

Complementing the $L^2$-stability of the density in~\eqref{eq:primal-density-l2-bound}, we have the following \(L^\infty\)-stability if the perturbation of the cost function is quantified in $L^\infty$.

\begin{remark}[\(L^\infty\)-stability of the density]
\label{rem:primal-density-linfty-stability}
Let $\Delta_*:=d_{\mathfrak D}\big((P,Q,c,\eps),(P',Q',c',\eps')\big)$ be as in~\eqref{eq:distance-D} and $\overline C$ as in~\eqref{eq:overline-C}. Using the notation of \cref{thm:primal-stability}, but replacing \(\hat\Delta\) by 
\begin{equation*}
\hat\Delta_\infty
:=
\min\left\{
\frac{\overline C\,\Delta_*+\|c-c'\|_{L^\infty(\Gamma)}+A'|\eps-\eps'|}{\eps},
\frac{\overline C\,\Delta_*+\|c-c'\|_{L^\infty(\Gamma)}+A|\eps-\eps'|}{\eps'}
\right\},
\end{equation*}
we have the following \(L^\infty\) bound: if $\Delta_*<\overline\eta_*$, where \(\overline\eta_*\) is as in \cref{cor:solution-map-lipschitz}, then 
\[
\|\zeta-\zeta'\|_{L^\infty(\Gamma)}\le \hat\Delta_\infty.
\]
The proof is the same as for \cref{eq:primal-density-l2-bound}, except that now the application of \cref{cor:symmetric-l2-stability} is replaced by an application of \cref{cor:solution-map-lipschitz}.
\end{remark}

\section{Stability of the optimal support}
\label{sec:support-stability}

As the support of the optimal coupling is characterized by~\eqref{eq:optimal-support}, we expect that stability of the supports is linked to stability of the potentials. However, the latter alone is not sufficient. This is most evident in the degenerate example \(c\equiv 0\), where it is easy to see that the optimal coupling is the product $\pi=P\otimes Q$ of the marginals (for any $\eps>0$). In particular, the potentials satisfy $f\oplus g \equiv \eps$, independently of $(P,Q)$, showing that control of the potentials alone cannot imply control of the supports. For that reason, the stability result below features a direct control $\Delta_\Omega$ on the marginal supports on the right-hand side of~\eqref{eq:support-hausdorff-bound}, in addition to the quantities needed to control the potentials.\footnote{One could, of course, choose a metric such as $W_\infty$ to obtain simultaneous control, but that would yield a very loose bound. E.g., in a case where the marginals vary but their supports are fixed, our control $\Delta_\Omega$ is vacuous, but $W_\infty$ would give a strong restriction.}

Given compact sets \(\emptyset \neq A,B\subset\R^{m}\), their Hausdorff distance is denoted
\[
d_H(A,B)
:=
\max\Big\{
\sup_{z\in A}\dist(z,B),\,
\sup_{z\in B}\dist(z,A)
\Big\}.
\]
We recall that \(\dist\) and \(d_H\) are computed with respect to the Euclidean norm, i.e., for \(z=(x,y)\in\R^d\times\R^d\), we use \(\|z\|=(\|x\|^2+\|y\|^2)^{1/2}\). 

To guarantee stability of the optimal supports, the next theorem requires the nondegeneracy condition~\eqref{eq:exterior-nondegeneracy} which will be discussed in detail below. In particular, we will see in \cref{pr:concavity-exterior-nondegeneracy} that~\eqref{eq:exterior-nondegeneracy} is satisfied with $a=\eps/\diam(\Omega_P)$ for the quadratic cost $c(x,y)=\frac12\|x-y\|^2$. On the other hand, \cref{ex:failure-support-stability} will illustrate that, in general, stability can fail when~\eqref{eq:exterior-nondegeneracy} is not satisfied.

\begin{theorem}[Support stability]
\label{thm:support-stability}
Let \(\mathfrak D=\mathfrak D(\underline\eps,D,L,\underline\lambda,\overline\Lambda,\underline\delta,\underline q)\) be as in \cref{def:admissible-class} and $\overline C$, $\overline\eta_*$ as in \cref{cor:solution-map-lipschitz}. For \((P,Q,c,\eps)\in\mathfrak D\), let
$h_{(P,Q,c,\eps)}=f\oplus g$, where $(f,g)$ are the potentials
corresponding to \((P,Q,c,\eps)\). Let
\[
(P,Q,c,\eps),\ (P',Q',c',\eps')\in\mathfrak D, \qquad \Delta_*:=d_{\mathfrak D}((P,Q,c,\eps),(P',Q',c',\eps')).
\]
Set
\[
\Delta_\Omega:=\Big(d_H(\Omega_P,\Omega_{P'})^2+d_H(\Omega_Q,\Omega_{Q'})^2\Big)^{1/2},
\]
\[
h:=h_{(P,Q,c,\eps)},
\qquad
h':=h_{(P',Q',c',\eps')},
\]
\[
\sigma:=h-c,
\qquad
\sigma':=h'-c',
\qquad
\pi:=\pi_{(P,Q,c,\eps)},
\qquad
\pi':=\pi_{(P',Q',c',\eps')},
\]
\[
\Sigma:=\supp(\pi),
\qquad
\Sigma':=\supp(\pi'),
\qquad
\delta_*:=\overline C\,\Delta_*+\|c-c'\|_{L^\infty(\X\times\Y)}.
\]
Assume that there exists \(a>0\) such that
\begin{equation}
\label{eq:exterior-nondegeneracy}
\sigma(z)\le -a\,\dist(z,\Sigma)
\qquad\text{for all }z\in \Omega_{P\otimes Q}\setminus\Sigma,
\end{equation}
and that the same bound holds for $(\sigma',\Sigma',P',Q')$, with the same constant $a$.
If $\Delta_*<\overline\eta_*$, the Hausdorff distance of the optimal supports satisfies
\begin{equation}
\label{eq:support-hausdorff-bound}
d_H(\Sigma,\Sigma')
\le
\left(1+\frac{(\sqrt2+1)L}{a}\right)\Delta_\Omega+\frac{\delta_*}{a}.
\end{equation}
\end{theorem}

Under the nondegeneracy assumption~\eqref{eq:exterior-nondegeneracy}, the value \(\sigma(z)\) is comparable to $-\dist(z,\Sigma)$ outside the support $\Sigma=\overline{\{z\in\Omega_{P\otimes Q}:\sigma(z)>0\}}$:
\[
  -\Lip(\sigma) \dist(z,\Sigma) \leq \sigma(z) \leq -a \dist(z,\Sigma)\qquad\text{for all }z\in \Omega_{P\otimes Q}\setminus\Sigma,
\]
as $\sigma\geq0$ on $\Sigma$. This linear detachment prevents instability. In general, if $\sigma$ is nearly flat at the boundary of the support, a small perturbation of $\sigma$ can cause a large change of the support. Indeed, the following example, where the marginals vary without changing support and $c,\eps$ are fixed, shows that support stability for the optimal coupling can fail if the nondegeneracy condition~\eqref{eq:exterior-nondegeneracy} is removed.

\begin{example}[Failure of support stability]
\label{ex:failure-support-stability}
Let
\[
\X=[0,1],\qquad \Y=\{0,1\}, \qquad \eps=1.
\]
Define
\[
u(x)=
\begin{cases}
0, & 0\le x\le \frac14,\\[1mm]
\frac{32}{5}\left(x-\frac14\right), & \frac14<x<\frac12,\\[1mm]
\frac85, & \frac12\le x\le 1.
\end{cases}
\]
Then \(u\) is Lipschitz, \(0\le u\le 8/5<2\), and $\int_0^1 u(x)\,dx=1$. We consider the cost function
\[
c(x,0)=u(x),
\qquad
c(x,1)=2-u(x),
\]
and note that $c$ can be extended to a Lipschitz function on $\R\times\R$. Moreover, we define a family of marginals indexed by \(0\leq\eta<1\) as follows. The second marginal is fixed to be
\[
Q^\eta=Q^0=\frac12\delta_0+\frac12\delta_1,
\]
whereas the first marginal \(P^\eta\) is defined via its density
\[
p_\eta(x):=\frac{dP^\eta}{dx}(x)
=
\begin{cases}
1+\eta, & 0\le x\le \frac14,\\[1mm]
1-\frac{\eta}{3}, & \frac14<x\le 1.
\end{cases}
\]
Then \(P^\eta\) has density bounded above and below on the fixed convex compact
support \([0,1]\), \(P^0\) is the uniform measure on \([0,1]\), and \(P^\eta\to P^0\) in total variation.

For the marginals \(P^0,Q^0\), the sum of the potentials is given by
\[
h^0(x,0)=h^0(x,1)=2.
\]
Indeed, this gives rise to 
\[
\sigma^0(x,0)=h^0(x,0)-c(x,0)=2-u(x),
\qquad
\sigma^0(x,1)=h^0(x,1)-c(x,1)=u(x).
\]
In view of
$
\frac12(\sigma^0(x,0)+\sigma^0(x,1))=1
$ and $
\int_0^1 \sigma^0(x,1)\,dx=\int_0^1 u(x)\,dx=1,
$
the plan with density \((\sigma^0)_+\) with respect to \(P^0\otimes Q^0\)
satisfies the marginal constraints and is optimal. Its support is
\[
\Sigma^0
=
\supp(\pi^0)
=
\big([0,1]\times\{0\}\big)
\cup
\big([\tfrac14,1]\times\{1\}\big).
\]
We note that the nondegeneracy condition \eqref{eq:exterior-nondegeneracy} fails: as
$
\sigma^0(0,1)=0
$
and
$
\dist\big((0,1),\Sigma^0\big)=\frac14,
$
there is no \(a>0\) such that
\[
\sigma^0(z)\le -a\,\dist(z,\Sigma^0)
\qquad
\text{for all }z\in\Omega_{P^0\otimes Q^0}\setminus\Sigma^0.
\]

Now consider \(P^\eta,Q^\eta\). In view of $\int_0^1 u(x)p_\eta(x)\,dx=1-\frac{\eta}{3}$, putting 
\[
h^\eta(x,0)=2-\delta_\eta,
\qquad
h^\eta(x,1)=2+\delta_\eta,
\qquad \delta_\eta=\frac{\eta}{3}
\]
gives
\[
\sigma^\eta(x,0)
=
2-u(x)-\delta_\eta,
\qquad
\sigma^\eta(x,1)
=
u(x)+\delta_\eta.
\]
For \(0<\eta<1\), both quantities are positive on the whole of \([0,1]\).
Moreover,
$
\frac12\big(\sigma^\eta(x,0)+\sigma^\eta(x,1)\big)=1
$
and
$
\int_0^1 \sigma^\eta(x,1)\,dP^\eta(x)
=
\int_0^1 \big(u(x)+\delta_\eta\big)p_\eta(x)\,dx
=
1.
$
Thus the optimal plan for \((P^\eta,Q^\eta,c,1)\) has density
\((\sigma^\eta)_+\) with respect to \(P^\eta\otimes Q^\eta\), and therefore
\[
\Sigma^\eta
=
\supp(\pi^\eta)
=
[0,1]\times\{0,1\}.
\]
Consequently, $\Sigma^\eta$ does not converge to $\Sigma^0$:
\[
d_H(\Sigma^\eta,\Sigma^0)
=
\frac14
\qquad
\text{for every }0<\eta<1.
\]
On the other hand, $P^\eta\to P^0$ and $Q^\eta=Q^0$, and the marginal supports $\Omega_{P^\eta}=\Omega_{P^0}=[0,1]$ and $\Omega_{Q^\eta}=\Omega_{Q^0}=\{0,1\}$ do not move, so that all conditions of \cref{thm:support-stability} except~\eqref{eq:exterior-nondegeneracy} are satisfied for sufficiently small $\eta>0$.
\end{example}

To apply \cref{thm:support-stability}, we need to verify the nondegeneracy condition~\eqref{eq:exterior-nondegeneracy}. The next proposition shows that for the key example of quadratic cost, the condition holds as soon as one marginal support is convex and compact.

\begin{proposition}[Nondegeneracy]
\label{pr:concavity-exterior-nondegeneracy}
Let \((P,Q,c,\eps)\) satisfy the hypotheses of \cref{prop:qot-basic-properties} and assume that \(\Omega_P\) is convex with \(D_P:=\diam(\Omega_P)<\infty\). Set
\[
h:=h_{(P,Q,c,\eps)},
\qquad
\sigma:=h-c,
\qquad
\pi:=\pi_{(P,Q,c,\eps)},
\qquad
\Sigma:=\supp(\pi).
\]
Assume that for each \(y\in\Omega_Q\), the map
\begin{equation}
\label{eq:fiberwise-concavity}
x\mapsto \sigma(x,y) \qquad\text{is concave on } \Omega_P.
\end{equation}
If \(D_P=0\), then \(\Omega_{P\otimes Q}\setminus\Sigma=\emptyset\); if \(D_P>0\), then
\begin{equation}
\label{eq:concavity-exterior-nondegeneracy}
\sigma(z)\le -\frac{\eps}{D_P}\,\dist(z,\Sigma)
\qquad\text{for all }z\in \Omega_{P\otimes Q}\setminus\Sigma,
\end{equation}
in particular, the nondegeneracy condition~\eqref{eq:exterior-nondegeneracy} holds with $a=\eps/D_P$. 

The assumption~\eqref{eq:fiberwise-concavity} holds in particular for the quadratic cost $c(x,y)=\frac12\|x-y\|^2$.
\end{proposition}

Combining \cref{thm:support-stability} and \cref{pr:concavity-exterior-nondegeneracy} yields the following corollary, where we also specialize to $\eps=\eps'$ to simplify the expressions.

\begin{corollary}[Support stability for quadratic cost]
\label{cor:quadratic-support-stability}
Let \(\mathfrak D=\mathfrak D(\eps,D,L,\underline\lambda,\overline\Lambda,\underline\delta,\underline q)\) be as in \cref{def:admissible-class} and $\overline C$, $\overline\eta_*$ as in \cref{cor:solution-map-lipschitz}. Let $\X,\Y$ be such that the quadratic cost
$
c(x,y):=\frac12\|x-y\|^2
$
is \(L\)-Lipschitz on \(\X\times\Y\). Consider marginals $(P,Q)$ and $(P',Q')$ such that 
$(P,Q,c,\eps)$ and $(P',Q',c,\eps)$ belong to $\mathfrak D$. Set
\[
\Delta_W
:=
\Big(W_1(P,P')^2+W_1(Q,Q')^2\Big)^{1/2},
\qquad
\Delta_\Omega
:=
\Big(
d_H(\Omega_P,\Omega_{P'})^2
+
d_H(\Omega_Q,\Omega_{Q'})^2
\Big)^{1/2}.
\]
If \(2L\Delta_W<\overline\eta_*\), the corresponding optimal supports $\Sigma,\Sigma'$ satisfy
\begin{equation*}
d_H(\Sigma,\Sigma')
\le
\left(
1+\frac{(\sqrt2+1)L D}{\eps}
\right)\Delta_\Omega
+
\frac{2L D\,\overline C}{\eps}\,\Delta_W .
\end{equation*}
\end{corollary}

\appendix
\section{Proofs}
\label{sec:proofs}

\subsection{Proofs for \cref{sec:setting-background}}

The basic structural properties of QOT in \cref{prop:qot-basic-properties} are known or follow by known arguments. For completeness, the following gives details and references.

\begin{proof}[Proof of \cref{prop:qot-basic-properties}]
By the McShane extension theorem, \(c\) admits an \(L\)-Lipschitz extension
to \(\R^d\times\R^d\), still denoted by \(c\). As this extension is of
linear growth and \(P,Q\) have finite second moments, we are in the setting
of~\cite{Nutz.2024} with \(c\in L^2(P\otimes Q)\). In particular,
\cite[Section~2]{Nutz.2024} shows~(i), (ii) with
\((f,g)\in L^1(P)\times L^1(Q)\), and~\eqref{eq:optimal-density}. It also yields the existence
of \(L\)-Lipschitz versions of \((f,g)\) on \((\Omega_P,\Omega_Q)\),
implying that \((f,g)\in L^2(P)\times L^2(Q)\). Moreover, it shows the
validity of the first-order condition~\eqref{eq:potential-foc} for all
\(x\in\Omega_P\) and \(y\in \Omega_Q\), respectively. We can then uniquely
extend \(f\) to \(\X\) by enforcing the first equation
of~\eqref{eq:potential-foc}, using the monotonicity of
\(a\mapsto (a+ g(y)-c(x,y))_+\) and the intermediate value theorem as in
the argument stated below \cite[Eq.~(2.11)]{Nutz.2024}. Similarly for
\(g\) on \(\Y\). As the proof of the \(L\)-Lipschitz property in
\cite[Section~2]{Nutz.2024} is based solely on the first-order condition, the
constructed extensions are also \(L\)-Lipschitz.

If one of the two marginal supports is connected, the proof
\cite[Lemma~3.1]{GonzalezSanzNutzRiveros.25} shows that \((f,g)\) are
unique on \((\Omega_P,\Omega_Q)\) up to translation. Uniqueness on \(\X\)
and \(\Y\) then follows by the aforementioned uniqueness of the extensions
satisfying~\eqref{eq:potential-foc}. The cited lemma is stated for
\(c(x,y)=\frac{1}{2}\|x-y\|^2\) on a compact domain, but the proof remains
valid without changes for uniformly continuous costs on \(\X\times\Y\).

Regarding (iv), the oscillation bounds are stated in
\cite[Lemma~2.5]{Nutz.2024}; again, the proof is based solely on the
first-order condition, hence the bounds hold without exceptional set for the
chosen versions. To infer \eqref{eq:potential-sum-bounds}, set
\[
C:=\|c\|_{L^\infty(\X\times\Y)},
\qquad
M:=\sup_{\X\times\Y}(f\oplus g),
\qquad
m:=\inf_{\X\times\Y}(f\oplus g).
\]
Since
\[
\osc_{\X\times\Y}(f\oplus g)
\le
\osc_{\X}(f)+\osc_{\Y}(g)
\le 2\,\osc_{\X\times\Y}(c)
\le 4C,
\]
we have
\[
M-m\le 4C.
\]
To prove the upper bound, note that for all \((x,y)\in \X\times\Y\),
\[
f(x)+g(y)-c(x,y)\ge M-4C-C=M-5C.
\]
Taking positive parts, integrating with respect to \(P\otimes Q\), using
\(\Omega_P\times\Omega_Q\subset\X\times\Y\), and~\eqref{eq:potential-foc}, we obtain
\[
\eps
=
\int (f\oplus g-c)_+\,d(P\otimes Q)
\ge (M-5C)_+,
\]
which implies \(M\le 5C+\eps\). For the lower bound, we similarly have
\[
f(x)+g(y)-c(x,y)\le m+4C+C=m+5C
\]
for all \((x,y)\in \X\times\Y\). Proceeding as above yields
$
0<\eps \le (m+5C)_+=m+5C,
$
that is, \(m\ge -5C+\eps\).

Finally, the support~\eqref{eq:optimal-support} readily follows from~\eqref{eq:optimal-density} and the continuity of $f,g,c$. 
\end{proof}

\subsection{Proofs for \cref{sec:l2-stability}}

For ease of reference, we record the following fact about decompositions of \(w\in\Hil\).

\begin{remark}\label{rem:balanced-decomposition}
For \(w\in\Hil\), the decomposition into \(w=u\oplus v\) is not unique. If
\(w=u\oplus v\) is any decomposition with
\((u,v)\in L^2(P)\times L^2(Q)\), then
\[
\|w\|_{L^2(P\otimes Q)}^2
=
\|u\|_{L^2(P)}^2+\|v\|_{L^2(Q)}^2
+2\Big(\int u\,dP\Big)\Big(\int v\,dQ\Big),
\]
and in particular
\begin{equation}\label{eq:centered-decomp-norm}
\|u\|_{L^2(P)}^2+\|v\|_{L^2(Q)}^2= \|w\|_{L^2(P\otimes Q)}^2 \qquad\text{if} \qquad\int u\,dP=0 \text{ or }\int v\,dQ = 0.
\end{equation}
Writing \(\bar w:=\int w\,d(P\otimes Q)\), a balanced choice of decomposition is
\begin{align}\label{eq:balanced-decomp}
u(x):=\int w(x,y)\,dQ(y)-\frac{\bar w}{2}, \qquad v(y):=\int w(x,y)\,dP(x)-\frac{\bar w}{2}.
\end{align}
Then,
\begin{equation}\label{eq:balanced-decomp-means}
\int u\,dP=\int v\,dQ=\frac{\bar w}{2},
\end{equation}
and in particular
\begin{equation}\label{eq:balanced-decomp-norm}
\|u\|_{L^2(P)}^2+\|v\|_{L^2(Q)}^2 \leq \|u\|_{L^2(P)}^2+\|v\|_{L^2(Q)}^2 + \frac{\bar w^2}{2} = \|w\|_{L^2(P\otimes Q)}^2.
\end{equation}
\end{remark}

Next, we prove the basic \(L^2\)-estimate for perturbations of the data.

\begin{proof}[Proof of \cref{thm:l2-stability}]
Since \(\Omega_P\) and \(\Omega_Q\) are compact, \(f'\) and \(g'\) are bounded by \cref{prop:qot-basic-properties}, 
and \cref{thm:dual-error-bound} applies to \(h'\). Moreover, by \cref{prop:qot-basic-properties},
\begin{equation}\label{eq:perturbed-potential-foc-l2}
\begin{cases}
\displaystyle \int (f'(x)+g'(y)-c'(x,y))_+\,dQ'(y)=\eps'
\quad \text{for all }x\in\X,\\[1.2ex]
\displaystyle \int (f'(x)+g'(y)-c'(x,y))_+\,dP'(x)=\eps'
\quad \text{for all }y\in\Y.
\end{cases}
\end{equation}

\medskip
\noindent\emph{Step 1: We show \(\|{\rm D}\Phi(h')\|_{L^2(P\otimes Q)}
\le
\Delta/\eps\).}
Let \(w\in\Hil\) satisfy \(\|w\|_{L^2(P\otimes Q)}\le 1\). Let
\((u,v)\) be the balanced decomposition \(w=u\oplus v\) given by
\eqref{eq:balanced-decomp}, and write
\(\bar w:=\int w\,d(P\otimes Q)\). Recalling~\eqref{eq:dual-gradient-action} and
using Fubini's theorem, we have
\begin{align}
\langle {\rm D}\Phi(h'),w\rangle_{\Hil}
&=
\int_{\Omega_P} u(x)\Big(1-\frac{1}{\eps}\int (h'(x,y)-c(x,y))_+\,dQ(y)\Big)\,dP(x)
\nonumber\\
&\quad
+\int_{\Omega_Q} v(y)\Big(1-\frac{1}{\eps}\int (h'(x,y)-c(x,y))_+\,dP(x)\Big)\,dQ(y).
\label{eq:l2-gradient-split}
\end{align}

For \(x\in\Omega_P\), define
\[
\xi'_x(y):=(h'(x,y)-c'(x,y))_+,
\qquad y\in\Y.
\]
Then \eqref{eq:perturbed-potential-foc-l2} implies
\[
\int \xi'_x(y)\,dQ'(y)=\eps'
\qquad \text{for all }x\in\Omega_P,
\]
and hence
\begin{multline}\label{eq:l2-q-residual}
1-\frac{1}{\eps}\int (h'(x,y)-c(x,y))_+\,dQ(y)
=
\frac{\eps-\eps'}{\eps}
+\frac{1}{\eps}\int \xi'_x(y)\,d(Q'-Q)(y)
\\
+\frac{1}{\eps}\int \big[(h'(x,y)-c'(x,y))_+-(h'(x,y)-c(x,y))_+\big]\,dQ(y).
\end{multline}
Similarly, for \(y\in\Omega_Q\), defining
$
\eta'_y(x):=(h'(x,y)-c'(x,y))_+,
$
we have
\begin{multline}\label{eq:l2-p-residual}
1-\frac{1}{\eps}\int (h'(x,y)-c(x,y))_+\,dP(x)
=
\frac{\eps-\eps'}{\eps}
+\frac{1}{\eps}\int \eta'_y(x)\,d(P'-P)(x)
\\
+\frac{1}{\eps}\int \big[(h'(x,y)-c'(x,y))_+-(h'(x,y)-c(x,y))_+\big]\,dP(x).
\end{multline}
Substituting \eqref{eq:l2-q-residual} and \eqref{eq:l2-p-residual} into
\eqref{eq:l2-gradient-split}, and using
\eqref{eq:balanced-decomp-means}, we obtain
\begin{align}
&\langle {\rm D}\Phi(h'),w\rangle_{\Hil} \nonumber\\
&=
\frac{\eps-\eps'}{\eps}\,\bar w
\label{eq:l2-pairing-epsilon}
\\
&\quad
+\frac{1}{\eps}\int_{\Omega_P\times\Omega_Q}
w(x,y)\Big((h'(x,y)-c'(x,y))_+-(h'(x,y)-c(x,y))_+\Big)\,d(P\otimes Q)(x,y)
\label{eq:l2-pairing-cost}
\\
&\quad
+\frac{1}{\eps}\int_{\Omega_P}
u(x)\Big(\int \xi'_x(y)\,d(Q'-Q)(y)\Big)\,dP(x)
\label{eq:l2-pairing-q-marginal}
\\
&\quad
+\frac{1}{\eps}\int_{\Omega_Q}
v(y)\Big(\int \eta'_y(x)\,d(P'-P)(x)\Big)\,dQ(y).
\label{eq:l2-pairing-p-marginal}
\end{align}

Next, we estimate each term. For \eqref{eq:l2-pairing-epsilon}, noting
\(|\bar w|\leq\|w\|_{L^2(P\otimes Q)}\le 1\), we have
\begin{equation}\label{eq:l2-epsilon-bound}
\left|\frac{\eps-\eps'}{\eps}\bar w\right|
\le
\frac{|\eps-\eps'|}{\eps}\|w\|_{L^2(P\otimes Q)}
\le \frac{|\eps-\eps'|}{\eps}.
\end{equation}
For \eqref{eq:l2-pairing-cost}, using that \(t\mapsto t_+\) is
\(1\)-Lipschitz and the Cauchy--Schwarz inequality,
\begin{equation}\label{eq:l2-cost-bound}
\left|\eqref{eq:l2-pairing-cost}\right|
\le
\frac{\|c-c'\|_{L^2(P\otimes Q)}}{\eps}\|w\|_{L^2(P\otimes Q)}
\le \frac{\|c-c'\|_{L^2(P\otimes Q)}}{\eps}.
\end{equation}
In \eqref{eq:l2-pairing-q-marginal}, the function \(\xi'_x\) is
\(2L\)-Lipschitz, uniformly in \(x\), because \(g'\) and \(c'\)
are \(L\)-Lipschitz and \(t\mapsto t_+\) is \(1\)-Lipschitz. Thus, by
Kantorovich--Rubinstein duality,
\[
\left|\int \xi'_x(y)\,d(Q'-Q)(y)\right|
\le 2L\,W_1(Q,Q')
\qquad \text{for all }x\in\Omega_P,
\]
showing
\begin{equation}\label{eq:l2-q-marginal-bound}
\left|\eqref{eq:l2-pairing-q-marginal}\right|
\le
\frac{2L}{\eps}W_1(Q,Q')\,\|u\|_{L^2(P)}.
\end{equation}
Similarly,
\begin{equation}\label{eq:l2-p-marginal-bound}
\left|\eqref{eq:l2-pairing-p-marginal}\right|
\le
\frac{2L}{\eps}W_1(P,P')\,\|v\|_{L^2(Q)}.
\end{equation}
Write
\(\Delta_W:=[W_1(P,P')^2+W_1(Q,Q')^2]^{1/2}\). Combining
\eqref{eq:l2-q-marginal-bound}, \eqref{eq:l2-p-marginal-bound},
Cauchy--Schwarz, and~\eqref{eq:balanced-decomp-norm},
\begin{align}
\left|\eqref{eq:l2-pairing-q-marginal}+\eqref{eq:l2-pairing-p-marginal}\right|
&\le
\frac{2L}{\eps}\Big(
W_1(Q,Q')\|u\|_{L^2(P)}+W_1(P,P')\|v\|_{L^2(Q)}
\Big)
\nonumber\\
&\le
\frac{2L}{\eps}\Delta_W
\Big(\|u\|_{L^2(P)}^2+\|v\|_{L^2(Q)}^2\Big)^{1/2}
\nonumber\\
&\le
\frac{2L}{\eps}\Delta_W\,\|w\|_{L^2(P\otimes Q)}
\le \frac{2L}{\eps}\Delta_W.
\label{eq:l2-marginal-bound}
\end{align}

Collecting \eqref{eq:l2-epsilon-bound}, \eqref{eq:l2-cost-bound}, and
\eqref{eq:l2-marginal-bound}, we obtain
\[
|\langle {\rm D}\Phi(h'),w\rangle_{\Hil}|
\le
\frac{1}{\eps}\big(2L\,\Delta_W+\|c-c'\|_{L^2(P\otimes Q)}+|\eps-\eps'|\big)
=
\frac{\Delta}{\eps}.
\]
Since \(w\) was arbitrary in the unit ball of \(\Hil\),
\eqref{eq:dual-gradient-norm} yields
\begin{equation}\label{eq:l2-gradient-bound}
\|{\rm D}\Phi(h')\|_{L^2(P\otimes Q)}
\le
\frac{\Delta}{\eps}.
\end{equation}

\medskip
\noindent\emph{Step 2: Proof of (i)--(iii).}
Applying the error bound from \cref{thm:dual-error-bound} to \(h'\), and using
\eqref{eq:l2-gradient-bound}, we obtain~\eqref{eq:l2-local-error-bound},
completing the proof of~(i). To show (ii), let \(\delta>0\) and assume the
left-hand side of \eqref{eq:l2-linfty-modulus} holds. Set
\[
\theta:=h-h',
\qquad
m:=\|\theta\|_{L^\infty(\Omega_P\times\Omega_Q)},
\]
and suppose for contradiction that \(m\ge \delta\). The function
\(\theta\) is continuous on the compact set \(\Omega_P\times\Omega_Q\), so there
exists \((x_0,y_0)\in\Omega_P\times\Omega_Q\) with \(|\theta(x_0,y_0)|=m\).
Replacing \(\theta\) by \(-\theta\) if necessary, we may assume that
\[
\theta(x_0,y_0)=m.
\]
Let
\[
\varrho:=\frac{\delta}{8L}.
\]
For all \(x\in B_\varrho(x_0)\cap\Omega_P\) and
\(y\in B_\varrho(y_0)\cap\Omega_Q\), the \(L\)-Lipschitz continuity of
\(f,f',g,g'\) yields
\begin{align*}
\theta(x,y)
&\ge
\theta(x_0,y_0)-|f(x)-f(x_0)|-|f'(x)-f'(x_0)|
-|g(y)-g(y_0)|-|g'(y)-g'(y_0)|\\
&\ge
m-2L\|x-x_0\|-2L\|y-y_0\|
\ge
m-\frac{\delta}{2}
\ge
\frac{m}{2}.
\end{align*}
Therefore,
\begin{align*}
\|\theta\|_{L^2(P\otimes Q)}^2
&\ge
\frac{m^2}{4}
(P\otimes Q)\big((B_\varrho(x_0)\cap\Omega_P)\times(B_\varrho(y_0)\cap\Omega_Q)\big)
\\
&=
\frac{m^2}{4}
P\big(B_\varrho(x_0)\big)Q\big(B_\varrho(y_0)\big)
\\
&\ge
\frac{m^2}{4}
\delta_P\min(\varrho^d,1)
\inf_{y\in\Omega_Q}Q\big(B_\varrho(y)\big)
=
\frac{\vartheta_\delta}{4}m^2,
\end{align*}
or equivalently,
\begin{equation}\label{eq:l2-supnorm-bootstrap}
m\le \frac{2}{\sqrt{\vartheta_\delta}}\,
\|\theta\|_{L^2(P\otimes Q)}.
\end{equation}

On the other hand, \eqref{eq:l2-local-error-bound} gives
\begin{equation}\label{eq:l2-error-bound-reused}
\|\theta\|_{L^2(P\otimes Q)}
\le
\frac{\gamma_\eps}{\eps}\max(m,\eps)\,\Delta.
\end{equation}
Moreover, distinguishing the cases \(\delta\ge \eps\) and \(\delta\le \eps\),
we see that \(m\ge \delta\) implies
\begin{equation}\label{eq:l2-max-reduction}
\max(m,\eps)\le \frac{\eps}{\min(\delta,\eps)}\,m.
\end{equation}
Combining \eqref{eq:l2-supnorm-bootstrap}, \eqref{eq:l2-error-bound-reused}, and
\eqref{eq:l2-max-reduction}, we obtain
\[
\|\theta\|_{L^2(P\otimes Q)}
\le
\frac{2\gamma_\eps}{\min(\delta,\eps)\sqrt{\vartheta_\delta}}
\,\Delta\,\|\theta\|_{L^2(P\otimes Q)}.
\]
As \(\|\theta\|_{L^2(P\otimes Q)}>0\), this contradicts the left-hand side of
\eqref{eq:l2-linfty-modulus}. Therefore,
\(m<\delta\), completing the proof of~(ii). Finally, (iii) follows by taking
\(\delta=\eps\) in (ii) and applying~(i).
\end{proof}

\medskip

We easily derive the symmetric \(L^2\)-stability statement.
\begin{proof}[Proof of \cref{cor:symmetric-l2-stability}]
Applying \cref{thm:l2-stability}(iii) to \((P,Q,c,\eps)\) and \((P',Q',c',\eps')\) yields \eqref{eq:symmetric-l2-bound-unprimed} and \eqref{eq:symmetric-l2-bound-primed}, respectively. It follows that
\[
\|h-h'\|_{L^2(\bar\mu)}^2
\le
\frac{\overline\gamma^2}{2}\bigl(\Delta^2+(\Delta')^2\bigr).
\]
It is elementary to show that 
$
\frac{1}{2}\bigl(\Delta^2+(\Delta')^2\bigr)
\le
\bar\Delta^2,
$ so that \eqref{eq:symmetric-l2-bound-mixture} follows.
\end{proof}

\MNg{
\begin{proof}[Proof of \cref{cor:symmetric-l2-stability}]
We first apply \cref{thm:l2-stability}(iii) with reference datum
\((P,Q,c,\eps)\). For this application, we choose the structural constants
\[
\lambda_P:=\underline\lambda,
\qquad
\Lambda_P:=\overline\Lambda,
\qquad
\delta_P:=\underline\delta.
\]
This is legitimate by \cref{def:admissible-class}(d)--(e). Moreover, by
\cref{def:admissible-class}(a), (c), (f),
\begin{align*}
\diam(\Omega_P)&\le D,
\\
\inf_{y\in\Omega_Q}Q\Big(B_{\eps/(8L)}(y)\Big)
&\ge
\inf_{y\in\Omega_Q}Q\Big(B_{\underline\eps/(8L)}(y)\Big)
\ge \underline q.
\end{align*}
Hence the constants in \cref{thm:l2-stability} satisfy
\[
\gamma_\eps\le \overline\gamma,
\qquad
\vartheta_\eps\ge \underline\vartheta.
\]
Therefore,
\[
\Delta<\overline\eta
=
\frac{\underline\eps\sqrt{\underline\vartheta}}{2\overline\gamma}
\le
\frac{\eps\sqrt{\vartheta_\eps}}{2\gamma_\eps},
\]
and \cref{thm:l2-stability}(iii) yields
\[
\|h-h'\|_{L^2(P\otimes Q)}
\le
\gamma_\eps\,\Delta
\le
\overline\gamma\,\Delta.
\]
This proves \eqref{eq:symmetric-l2-bound-unprimed}.

Interchanging primed and unprimed quantities and repeating the same argument
gives \eqref{eq:symmetric-l2-bound-primed}.

Finally,
\[
\|h-h'\|_{L^2(\bar\mu)}^2
=
\frac12\|h-h'\|_{L^2(P\otimes Q)}^2
+
\frac12\|h-h'\|_{L^2(P'\otimes Q')}^2,
\]
so that \eqref{eq:symmetric-l2-bound-unprimed} and
\eqref{eq:symmetric-l2-bound-primed} imply
\[
\|h-h'\|_{L^2(\bar\mu)}^2
\le
\frac{\overline\gamma^2}{2}\bigl(\Delta^2+(\Delta')^2\bigr).
\]
Writing
\[
A:=2L\,\Delta_W+|\eps-\eps'|,
\qquad
a:=\|c-c'\|_{L^2(P\otimes Q)},
\qquad
b:=\|c-c'\|_{L^2(P'\otimes Q')},
\]
we have
\[
\Delta=A+a,
\qquad
\Delta'=A+b,
\qquad
\bar\Delta=A+\Big(\frac{a^2+b^2}{2}\Big)^{1/2}.
\]
Since
\[
\frac{\Delta^2+(\Delta')^2}{2}
=
A^2+A(a+b)+\frac{a^2+b^2}{2}
\le
A^2+2A\Big(\frac{a^2+b^2}{2}\Big)^{1/2}+\frac{a^2+b^2}{2}
=
\bar\Delta^2,
\]
it follows that
\[
\|h-h'\|_{L^2(\bar\mu)}
\le
\overline\gamma\,\bar\Delta,
\]
which is \eqref{eq:symmetric-l2-bound-mixture}.
\end{proof}
}
\subsection{Proofs for \cref{sec:Linfty-stability}}

We can now upgrade the \(L^2\)-estimate to an \(L^\infty\)-estimate for the potentials.

\begin{proof}[Proof of \cref{thm:linfty-stability}]
The positivity of \(\widehat q_\eps,\vartheta_\eps\) is stated in
\cref{rem:ball-measure-lower-bounds}, and
\(\widehat\kappa_\eps>0\) is immediate. Hence, \(\widehat\eta_\eps>0\). Again, by \cref{prop:qot-basic-properties},
\begin{equation}\label{eq:perturbed-potential-foc-linfty}
\begin{cases}
\displaystyle \int (f'(x)+g'(y)-c'(x,y))_+\,dQ'(y)=\eps'
\quad \text{for all }x\in\X,\\[1.2ex]
\displaystyle \int (f'(x)+g'(y)-c'(x,y))_+\,dP'(x)=\eps'
\quad \text{for all }y\in\Y.
\end{cases}
\end{equation}
Clearly, \eqref{eq:linfty-h-bound} follows from \eqref{eq:linfty-f-bound} and \eqref{eq:linfty-g-bound}. Steps 1--3 below prove \eqref{eq:linfty-f-bound}, and then Step~4 explains how a variant of those arguments proves \eqref{eq:linfty-g-bound}.

\smallskip
\noindent
\emph{Step 1: \(L^2\)-control.}
As
$
\|c-c'\|_{L^2(P\otimes Q)}
\le
\|c-c'\|_{L^\infty(\X\times\Y)},
$
the condition \eqref{eq:linfty-smallness} implies that the quantity $\Delta$ of
\cref{thm:l2-stability} satisfies 
$
\Delta \leq \Delta_*<\frac{\eps\sqrt{\vartheta_\eps}}{2\gamma_\eps},
$
and therefore \cref{thm:l2-stability} yields
\begin{equation}\label{eq:linfty-l2-input-bound}
\|h-h'\|_{L^2(P\otimes Q)}\le \gamma_\eps\,\Delta\le \gamma_\eps\,\Delta_*.
\end{equation}

Next, fix the additive constant in \((f',g')\). By changing \((f',g')\) to $(f'+a,g'-a)$ for $a:=\int (g'-g)\,dQ$, we may assume without loss of generality that
\[
  \int (g'-g)\,dQ=0.
\]
Set
\[
u:=f-f',
\qquad
v:=g-g'.
\]
Then
$
\int v\,dQ=0,
$
and so \eqref{eq:centered-decomp-norm} yields
\begin{align*}
\|h-h'\|_{L^2(P\otimes Q)}^2
&=
\|u\|_{L^2(P)}^2+\|v\|_{L^2(Q)}^2.
\end{align*}
As a consequence, \eqref{eq:linfty-l2-input-bound} implies
\begin{equation}\label{eq:linfty-l2-component-bounds}
\|u\|_{L^2(P)}\le \gamma_\eps\,\Delta_*,
\qquad
\|v\|_{L^2(Q)}\le \gamma_\eps\,\Delta_*.
\end{equation}

\smallskip
\noindent
\emph{Step 2: Slope estimate.}
Fix \(x\in\X\), and define
\[
F_x(a):=\int (a+g(y)-c(x,y))_+\,dQ(y),
\qquad a\in\R.
\]
By the first identity in the first-order condition~\eqref{eq:potential-foc},
\[
F_x(f(x))=\eps.
\]
Since \(y\mapsto (f(x)+g(y)-c(x,y))_+\) is continuous on the compact set
\(\Omega_Q\), there exists
\[
y(x)\in\argmax_{y\in\Omega_Q}(f(x)+g(y)-c(x,y))_+.
\]
Using \eqref{eq:potential-foc} once more, we have
\[
(f(x)+g(y(x))-c(x,y(x)))_+
\ge
\int (f(x)+g(y)-c(x,y))_+\,dQ(y)
=
\eps,
\]
and hence
\[
f(x)+g(y(x))-c(x,y(x))\ge \eps.
\]
Now let \(a,b\in[f(x)-\eps/2,f(x)+\eps/2]\) with \(a\le b\), and let
\(y\in B_{\eps/(4L)}(y(x))\cap\Omega_Q\). Since the map
\(y\mapsto g(y)-c(x,y)\) is \(2L\)-Lipschitz, we obtain
\[
a+g(y)-c(x,y)
\ge
f(x)-\frac{\eps}{2}+g(y(x))-c(x,y(x))-2L\|y-y(x)\|
\ge
\eps-\frac{\eps}{2}-\frac{\eps}{2}
=0.
\]
Therefore,
\[
(b+g(y)-c(x,y))_+-(a+g(y)-c(x,y))_+=b-a
\]
for all such \(y\). As the left-hand side is clearly nonnegative for any $y\in\Y$, we deduce that
\[
(b+g(y)-c(x,y))_+-(a+g(y)-c(x,y))_+ \geq (b-a) \1_{B_{\eps/(4L)}(y(x))\cap\Omega_Q}(y)
\]
for all $y\in\Y$. Integrating against \(Q(dy)\) and recalling $Q(B_{\eps/(4L)}(y(x)))\geq \widehat q_\eps$, we conclude
\begin{equation}\label{eq:fx-local-monotonicity}
F_x(b)-F_x(a)\ge \widehat q_\eps\,(b-a)
\qquad
\text{for all }a\le b\text{ in }[f(x)-\eps/2,f(x)+\eps/2].
\end{equation}
We claim that this implies
\begin{equation}\label{eq:fx-local-slope}
|F_x(a)-\eps|\ge \widehat q_\eps |a-f(x)|
\qquad\text{whenever }|a-f(x)|\le \eps/2.
\end{equation}
Indeed, if \(a\le f(x)\), then \eqref{eq:fx-local-monotonicity} with
\((a,b)=(a,f(x))\) yields
\[
\eps-F_x(a)=F_x(f(x))-F_x(a)\ge \widehat q_\eps(f(x)-a),
\]
whereas if  \(a\ge f(x)\), then \eqref{eq:fx-local-monotonicity} with
\((a,b)=(f(x),a)\) yields
\[
F_x(a)-\eps=F_x(a)-F_x(f(x))\ge \widehat q_\eps(a-f(x)).
\]

\smallskip
\noindent
\emph{Step 3: The bound for $f-f'$.}
Fix \(x\in\X\). The first identity in
\eqref{eq:perturbed-potential-foc-linfty} yields
\begin{align}
|F_x(f'(x))-\eps|
&\le
\left|
\int \Big[(f'(x)+g(y)-c(x,y))_+
-(f'(x)+g'(y)-c(x,y))_+\Big]\,dQ(y)
\right|
\nonumber\\
&\quad
+
\left|
\int \Big[(f'(x)+g'(y)-c(x,y))_+
-(f'(x)+g'(y)-c'(x,y))_+\Big]\,dQ(y)
\right|
\nonumber\\
&\quad
+
\left|
\int (f'(x)+g'(y)-c'(x,y))_+\,d(Q-Q')(y)
\right|
+
|\eps-\eps'|.
\label{eq:fx-residual-decomposition}
\end{align}
Since \(t\mapsto t_+\) is \(1\)-Lipschitz, the first term in \eqref{eq:fx-residual-decomposition} is bounded by
\[
\int |g(y)-g'(y)|\,dQ(y)\le \|v\|_{L^2(Q)}.
\]
The second term is bounded by
\[
\int |c(x,y)-c'(x,y)|\,dQ(y)\le \|c-c'\|_{L^\infty(\X\times\Y)}
\]
as \(x\in\X\) and \(\Omega_Q\subset\Y\). Finally,
\[
y\mapsto (f'(x)+g'(y)-c'(x,y))_+
\]
is \(2L\)-Lipschitz on \(\Y\), so the third term is bounded by
\(2L\,W_1(Q,Q')\). Using \eqref{eq:linfty-l2-component-bounds}, we obtain
\begin{equation}\label{eq:fx-residual-bound}
|F_x(f'(x))-\eps|
\le
(1+\gamma_\eps)\Delta_*<\frac{\eps\,\widehat q_\eps}{2},
\end{equation}
where the last inequality is due to \eqref{eq:linfty-smallness} and \eqref{eq:linfty-local-threshold}.

We claim that~\eqref{eq:fx-residual-bound} implies \(|f'(x)-f(x)|<\eps/2\). Indeed, if
\(f'(x)\ge f(x)+\eps/2\), then the monotonicity of \(F_x\), together with
\eqref{eq:fx-local-monotonicity}, yields
\[
F_x(f'(x))-\eps
\ge
F_x(f(x)+\eps/2)-F_x(f(x))
\ge
\frac{\eps\,\widehat q_\eps}{2},
\]
contradicting~\eqref{eq:fx-residual-bound}. If
\(f'(x)\le f(x)-\eps/2\), we instead have
\[
\eps - F_x(f'(x))
\ge
F_x(f(x))-F_x(f(x)-\eps/2)
\ge
\frac{\eps\,\widehat q_\eps}{2},
\]
again contradicting~\eqref{eq:fx-residual-bound}. Hence
\[
|f'(x)-f(x)|<\frac{\eps}{2}.
\]
Now \eqref{eq:fx-local-slope} applies with $a=f'(x)$, and using also the first inequality of~\eqref{eq:fx-residual-bound}, we have
\[
\widehat q_\eps\,|f'(x)-f(x)|
\le
|F_x(f'(x))-\eps|
\le
(1+\gamma_\eps)\Delta_*.
\]
Since \(x\in\X\) was arbitrary, we conclude that
\begin{equation*}
\|u\|_{L^\infty(\X)}\le (1+\gamma_\eps)\widehat q_\eps^{-1}\,\Delta_*,
\end{equation*}
completing the proof of \eqref{eq:linfty-f-bound}.

\smallskip
\noindent
\emph{Step 4: The bound for $g-g'$.}
Fix \(y\in\Y\), and define
\[
G_y(b):=\int (f(x)+b-c(x,y))_+\,dP(x),
\qquad b\in\R.
\]
Arguing as in Step 2, but now integrating against $P(dx)$ and using the constant $\widehat\kappa_\eps$, we obtain
\begin{equation*}
G_y(b)-G_y(a)\ge \widehat\kappa_\eps\,(b-a)
\qquad
\text{for all }a\le b\text{ in }[g(y)-\eps/2,g(y)+\eps/2],
\end{equation*}
and therefore
\begin{equation}\label{eq:gy-local-slope}
|G_y(b)-\eps|\ge \widehat\kappa_\eps |b-g(y)|
\qquad\text{whenever }|b-g(y)|\le \eps/2.
\end{equation}
Arguing as in Step~3, we further see that 
\begin{equation*}
|G_y(g'(y))-\eps|
\le
(1+\gamma_\eps)\Delta_*<\frac{\eps\,\widehat\kappa_\eps}{2},
\end{equation*}
which is analogous to~\eqref{eq:fx-residual-bound} but has $\widehat\kappa_\eps$ on the right-hand side. 
As before, this implies first that
$
|g'(y)-g(y)|<\frac{\eps}{2},
$
and then, by \eqref{eq:gy-local-slope},
\[
\widehat\kappa_\eps\,|g'(y)-g(y)|
\le
|G_y(g'(y))-\eps|
\le
(1+\gamma_\eps)\Delta_*.
\]
Since \(y\in\Y\) was arbitrary, we conclude that
\begin{equation*}
\|v\|_{L^\infty(\Y)}\le (1+\gamma_\eps)\widehat\kappa_\eps^{-1}\,\Delta_*,
\end{equation*}
proving \eqref{eq:linfty-g-bound}. 
\end{proof}

Finally, we make the constants uniform over the admissible class.

\begin{proof}[Proof of \cref{cor:solution-map-lipschitz}]
This follows directly from \cref{thm:linfty-stability} and the  definitions of the constants in \cref{cor:solution-map-lipschitz}.
\end{proof}
\MNg{
\begin{proof}[Proof of \cref{cor:solution-map-lipschitz}] 
Fix \((P,Q,c,\eps),(P',Q',c',\eps')\in\mathfrak D\), and abbreviate
\[
h:=h_{(P,Q,c,\eps)},
\qquad
h':=h_{(P',Q',c',\eps')}.
\]
We apply \cref{thm:linfty-stability} with
localization sets \(\X\) and \(\Y\).

For the reference datum \((P,Q,c,\eps)\), we choose in that theorem the
structural constants
\[
\lambda_P:=\underline\lambda,\qquad
\Lambda_P:=\overline\Lambda,\qquad
\delta_P:=\underline\delta.
\]
This is legitimate by assumptions (d) and (e). Moreover, by assumptions
(a), (c), (f), we have
\begin{align*}
\diam(\Omega_P)&\le D,
\\
\inf_{y\in\Omega_Q}Q\!\Big(B_{\eps/(8L)}(y)\Big)
&\ge
\inf_{y\in\Omega_Q}Q\!\Big(B_{\underline\eps/(8L)}(y)\Big)
\ge \underline q,
\\
\inf_{y\in\Omega_Q}Q\!\Big(B_{\eps/(4L)}(y)\Big)
&\ge
\inf_{y\in\Omega_Q}Q\!\Big(B_{\underline\eps/(8L)}(y)\Big)
\ge \underline q.
\end{align*}
Hence the constants in \cref{thm:linfty-stability}
satisfy
\begin{align}
\gamma_\eps
&=
16\left(
\delta_P^{-1}\max\Big(\frac{8L}{\eps},1\Big)^d
\right)
\frac{\Lambda_P^2}{\lambda_P^2}
\frac{\bigl(\lceil 8L\,{\rm diam}(\Omega_P)/\eps\rceil\bigr)^{d+2}}
{\inf_{y\in\Omega_Q} Q(B_{\eps/(8L)}(y))}
\le \overline\gamma,
\label{eq:uniform-gamma-upper}
\\
\vartheta_\eps
&=
\delta_P
\min\Big\{\Big(\frac{\eps}{8L}\Big)^d,1\Big\}
\inf_{y\in\Omega_Q}Q\!\Big(B_{\eps/(8L)}(y)\Big)
\ge \underline\vartheta,
\nonumber\\
\widehat q_\eps
&=
\inf_{y\in\Omega_Q}Q\!\Big(B_{\eps/(4L)}(y)\Big)
\ge \underline q,
\label{eq:uniform-qhat-lower}
\\
\widehat\kappa_\eps
&=
\delta_P\min\Big\{\Big(\frac{\eps}{4L}\Big)^d,1\Big\}
\ge \underline{\widehat\kappa}.
\label{eq:uniform-kappahat-lower}
\end{align}

Write
\begin{equation}\label{eq:solution-map-delta-star}
\Delta_*:=
d_{\mathfrak D}\big((P,Q,c,\eps),(P',Q',c',\eps')\big),
\end{equation}
as in \cref{thm:linfty-stability}. 
Assume that
\[
d_{\mathfrak D}\big((P,Q,c,\eps),(P',Q',c',\eps')\big)<\overline\eta_*.
\]
Then, by \cref{eq:uniform-linfty-constants},
\begin{align*}
\Delta_*
&<
\frac{\underline\eps\sqrt{\underline\vartheta}}{2\overline\gamma}
\le
\frac{\eps\sqrt{\vartheta_\eps}}{2\gamma_\eps},
\\
(1+\gamma_\eps)\Delta_*
&<
(1+\overline\gamma)\overline\eta_*
\le
\frac{\underline\eps\,\underline q}{2}
\le
\frac{\eps\,\widehat q_\eps}{2},
\\
(1+\gamma_\eps)\Delta_*
&<
(1+\overline\gamma)\overline\eta_*
\le
\frac{\underline\eps\,\underline{\widehat\kappa}}{2}
\le
\frac{\eps\,\widehat\kappa_\eps}{2}.
\end{align*}
Hence the assumptions of
\cref{thm:linfty-stability} are satisfied, and we
obtain
\[
\|h-h'\|_{L^\infty(\X\times\Y)}
\le
(1+\gamma_\eps)\bigl(\widehat q_\eps^{-1}+\widehat\kappa_\eps^{-1}\bigr)\Delta_*.
\]
Using \cref{eq:uniform-qhat-lower},
\eqref{eq:uniform-kappahat-lower}, \eqref{eq:uniform-gamma-upper}, and
\eqref{eq:solution-map-delta-star}, it follows that
\begin{align*}
\|h-h'\|_{L^\infty(\X\times\Y)}
\le
(1+\overline\gamma)\bigl(\underline q^{-1}+\underline{\widehat\kappa}^{-1}\bigr)
\,d_{\mathfrak D}\big((P,Q,c,\eps),(P',Q',c',\eps')\big),
\end{align*}
which is \cref{eq:uniform-linfty-lipschitz}. This proves the
corollary.
\end{proof}
}

\subsection{Proofs for \cref{sec:primal-stability}}

Next, we transfer the stability of potentials to stability of the primal optimizers.

\begin{proof}[Proof of \cref{thm:primal-stability}]
By \cref{prop:qot-basic-properties}(v),
\begin{equation}
\label{eq:primal-density-representation}
d\pi=\zeta\,d\mu,
\qquad
d\pi'=\zeta'\,d\mu'.
\end{equation}

\smallskip
\noindent
\emph{Step 1: Stability of the density.}
By \cref{cor:symmetric-l2-stability},
\begin{equation}
\label{eq:primal-potential-l2-bound}
\|h-h'\|_{L^2(\bar\mu)}
\le
\overline\gamma\,\bar\Delta.
\end{equation}
Moreover, \cref{prop:qot-basic-properties}(iv), applied to \((P,Q,c,\eps)\) and \((P',Q',c',\eps')\) with localization sets \(\Omega_P\cup\Omega_{P'}\) and \(\Omega_Q\cup\Omega_{Q'}\), yields
\[
-5\|c\|_{L^\infty(\Gamma)}+\eps
\le
h
\le
5\|c\|_{L^\infty(\Gamma)}+\eps,
\qquad
-5\|c'\|_{L^\infty(\Gamma)}+\eps'
\le
h'
\le
5\|c'\|_{L^\infty(\Gamma)}+\eps'
\]
on \(\Gamma\). Hence
\[
0\le (h-c)_+\le \eps A,
\qquad
0\le (h'-c')_+\le \eps' A'
\qquad\text{on }\Gamma.
\]

Using
\[
\zeta-\zeta'
=
\frac{(h-c)_+-(h'-c')_+}{\eps}
+
\Big(\frac{1}{\eps}-\frac{1}{\eps'}\Big)(h'-c')_+,
\]
the fact that \(t\mapsto t_+\) is \(1\)-Lipschitz, and \eqref{eq:primal-potential-l2-bound}, we obtain
\[
\|\zeta-\zeta'\|_{L^2(\bar\mu)}
\le
\frac{\|h-h'\|_{L^2(\bar\mu)}+\|c-c'\|_{L^2(\bar\mu)}}{\eps}
+
\frac{A'}{\eps}|\eps-\eps'|
\le
\frac{\overline\gamma\,\bar\Delta+\|c-c'\|_{L^2(\bar\mu)}+A'|\eps-\eps'|}{\eps}.
\]
Similarly,
\[
\zeta-\zeta'
=
\frac{(h-c)_+-(h'-c')_+}{\eps'}
+
\Big(\frac{1}{\eps}-\frac{1}{\eps'}\Big)(h-c)_+,
\]
and therefore
\[
\|\zeta-\zeta'\|_{L^2(\bar\mu)}
\le
\frac{\overline\gamma\,\bar\Delta+\|c-c'\|_{L^2(\bar\mu)}+A|\eps-\eps'|}{\eps'}.
\]
Taking the minimum of the two bounds yields \eqref{eq:primal-density-l2-bound}.

\smallskip
\noindent
\emph{Step 2: Control of the reference measures.} Writing \[
\mu-\mu'
=
(P-P')\otimes Q + P'\otimes(Q-Q'),
\]
one sees that 
\begin{equation}
\label{eq:product-tv-bound}
\|\mu-\mu'\|_{\mathrm{TV}}
\le
\|P-P'\|_{\mathrm{TV}} + \|Q-Q'\|_{\mathrm{TV}}
=
\Delta_{\mathrm{TV}}.
\end{equation}

Moreover, if \(\gamma_P\in\Pi(P,P')\) and \(\gamma_Q\in\Pi(Q,Q')\) are optimal couplings for \(W_1(P,P')\) and \(W_1(Q,Q')\), then \(\gamma_P\otimes\gamma_Q\) is a coupling of \(\mu\) and \(\mu'\). Therefore,
\begin{align}
W_1(\mu,\mu')
&\le
\int \|(x,y)-(x',y')\|\,d(\gamma_P\otimes\gamma_Q)(x,x',y,y')
\nonumber\\
&\le
\int \|x-x'\|\,d\gamma_P(x,x')
+
\int \|y-y'\|\,d\gamma_Q(y,y')
\nonumber\\
&=
W_1(P,P')+W_1(Q,Q')
\le
\sqrt2\,\Delta_W.
\label{eq:product-w1-bound}
\end{align}

\smallskip
\noindent
\emph{Step 3: TV-stability.}
Let
\[
u:=\frac{d\mu}{d\bar\mu},
\qquad
u':=\frac{d\mu'}{d\bar\mu}.
\]
Then \(u+u'=2\) \(\bar\mu\)-a.s., and by \eqref{eq:primal-density-representation},
\[
\frac{d\pi}{d\bar\mu}=\zeta\,u,
\qquad
\frac{d\pi'}{d\bar\mu}=\zeta'\,u'.
\]
Using \(u+u'=2\), we obtain
\begin{equation}
\label{eq:primal-density-radon-decomposition}
\frac{d\pi}{d\bar\mu}-\frac{d\pi'}{d\bar\mu}
=
\zeta u-\zeta' u'
=
(\zeta-\zeta')+\frac{\zeta+\zeta'}{2}(u-u').
\end{equation}
Hence,
\begin{align*}
2\|\pi-\pi'\|_{\mathrm{TV}}
&=
\left\|\frac{d\pi}{d\bar\mu}-\frac{d\pi'}{d\bar\mu}\right\|_{L^1(\bar\mu)}
\le
\|\zeta-\zeta'\|_{L^1(\bar\mu)}
+
\frac12\int (\zeta+\zeta')\,d|\mu-\mu'|.
\end{align*}
Since \(\bar\mu\) is a probability measure, \eqref{eq:primal-density-l2-bound} yields
\[
\|\zeta-\zeta'\|_{L^1(\bar\mu)}\le \hat\Delta.
\]
Moreover, by the definitions of \(A\) and \(A'\),
\[
0\le \zeta\le A,
\qquad
0\le \zeta'\le A'
\qquad\text{on }\Gamma,
\]
and \(\supp|\mu-\mu'|\subset \Theta\subset \Gamma\). Therefore,
\[
\frac12\int (\zeta+\zeta')\,d|\mu-\mu'|
\le
\frac{A+A'}{2}\,|\mu-\mu'|(\R^d\times\R^d)
=
(A+A')\|\mu-\mu'\|_{\mathrm{TV}}.
\]
Combining the preceding estimates with \eqref{eq:product-tv-bound}, we obtain
\[
2\|\pi-\pi'\|_{\mathrm{TV}}
\le
\hat\Delta+(A+A')\Delta_{\mathrm{TV}},
\]
which is \eqref{eq:primal-tv-bound}.

\smallskip
\noindent
\emph{Step 4: $W_1$-stability.}
Let \(\varphi:\Theta\to\R\) be \(1\)-Lipschitz. As subtracting a constant from \(\varphi\) does not change \(\int \varphi\,d(\pi-\pi')\), we may assume that
\[
\|\varphi\|_{L^\infty(\Theta)}\le \frac{D_*}{2}.
\]
Using \eqref{eq:primal-density-radon-decomposition}, we find
\begin{align*}
\int \varphi\,d(\pi-\pi')
&=
\int \varphi\,(\zeta-\zeta')\,d\bar\mu
+
\frac12\int \varphi\,(\zeta+\zeta')\,d(\mu-\mu').
\end{align*}
For the first term, \eqref{eq:primal-density-l2-bound} gives
\begin{align*}
\left|\int \varphi\,(\zeta-\zeta')\,d\bar\mu\right|
&\le
\frac{D_*}{2}\int |\zeta-\zeta'|\,d\bar\mu
\\
&\le
\frac{D_*}{2}\|\zeta-\zeta'\|_{L^2(\bar\mu)}
\le
\frac{D_*\,\hat\Delta}{2}.
\end{align*}
For the second term, define
\[
\psi:=\frac12\,\varphi\,(\zeta+\zeta')
\qquad\text{on }\Theta.
\]
Since \(\|\zeta+\zeta'\|_{L^\infty(\Theta)}\le A+A'\), and since \(h,h'\) are \(\sqrt2L\)-Lipschitz while \(c,c'\) are \(L\)-Lipschitz, we have
\[
\Lip(\zeta)
\le
\frac{(\sqrt2+1)L}{\eps},
\qquad
\Lip(\zeta')
\le
\frac{(\sqrt2+1)L}{\eps'}.
\]
Consequently,
\begin{align*}
\Lip(\psi)
&\le
\frac12\Lip(\varphi)\,\|\zeta+\zeta'\|_{L^\infty(\Theta)}
+
\frac12\|\varphi\|_{L^\infty(\Theta)}\,\Lip(\zeta+\zeta')
\\
&\le
\frac{A+A'}{2}
+
\frac{(\sqrt2+1)L\,D_*}{4}\Big(\frac{1}{\eps}+\frac{1}{\eps'}\Big).
\end{align*}
By Kantorovich--Rubinstein duality on the compact set \(\Theta\), and using \eqref{eq:product-w1-bound},
\begin{align*}
\left|\frac12\int \varphi\,(\zeta+\zeta')\,d(\mu-\mu')\right|
&=
\left|\int \psi\,d(\mu-\mu')\right|
\\
&\le
\Lip(\psi)\,W_1(\mu,\mu')
\\
&\le
\sqrt2\left[
\frac{A+A'}{2}
+
\frac{(\sqrt2+1)L\,D_*}{4}\Big(\frac{1}{\eps}+\frac{1}{\eps'}\Big)
\right]\Delta_W.
\end{align*}
Collecting the bounds for the two terms, we conclude that
\[
\left|\int \varphi\,d(\pi-\pi')\right|
\le
\frac{D_*\,\hat\Delta}{2}
+
\sqrt2\left[
\frac{A+A'}{2}
+
\frac{(\sqrt2+1)L\,D_*}{4}\Big(\frac{1}{\eps}+\frac{1}{\eps'}\Big)
\right]\Delta_W.
\]
Taking the supremum over all \(1\)-Lipschitz \(\varphi\) proves \eqref{eq:primal-w1-bound}.
\qedhere
\end{proof}

\MNg{
\begin{proof}[Proof of \cref{rem:primal-density-linfty-stability}]
Since \(\Gamma\subset \X\times\Y\), \cref{cor:solution-map-lipschitz} yields
\begin{equation}
\label{eq:primal-potential-linfty-bound}
\|h-h'\|_{L^\infty(\Gamma)}
\le
\|h-h'\|_{L^\infty(\X\times\Y)}
\le
\overline C\,\Delta_*.
\end{equation}

Next, \cref{prop:qot-basic-properties}(iv), applied to \((P,Q,c,\eps)\) with localization sets \(\Omega_P\cup\Omega_{P'}\) and \(\Omega_Q\cup\Omega_{Q'}\), yields
\[
-5\|c\|_{L^\infty(\Gamma)}+\eps
\le
h
\le
5\|c\|_{L^\infty(\Gamma)}+\eps
\qquad\text{on }\Gamma.
\]
Hence
\[
0\le (h-c)_+\le 6\|c\|_{L^\infty(\Gamma)}+\eps=\eps A
\qquad\text{on }\Gamma.
\]
Similarly,
\[
0\le (h'-c')_+\le 6\|c'\|_{L^\infty(\Gamma)}+\eps'=\eps' A'
\qquad\text{on }\Gamma.
\]

Using
\[
\zeta-\zeta'
=
\frac{(h-c)_+-(h'-c')_+}{\eps}
+
\Big(\frac{1}{\eps}-\frac{1}{\eps'}\Big)(h'-c')_+,
\]
the fact that \(t\mapsto t_+\) is \(1\)-Lipschitz, and \eqref{eq:primal-potential-linfty-bound}, we obtain
\[
\|\zeta-\zeta'\|_{L^\infty(\Gamma)}
\le
\frac{\|h-h'\|_{L^\infty(\Gamma)}+\|c-c'\|_{L^\infty(\Gamma)}}{\eps}
+
\frac{A'}{\eps}|\eps-\eps'|
\le
\frac{\overline C\,\Delta_*+\|c-c'\|_{L^\infty(\Gamma)}+A'|\eps-\eps'|}{\eps}.
\]
Similarly,
\[
\zeta-\zeta'
=
\frac{(h-c)_+-(h'-c')_+}{\eps'}
+
\Big(\frac{1}{\eps}-\frac{1}{\eps'}\Big)(h-c)_+,
\]
and therefore
\[
\|\zeta-\zeta'\|_{L^\infty(\Gamma)}
\le
\frac{\overline C\,\Delta_*+\|c-c'\|_{L^\infty(\Gamma)}+A|\eps-\eps'|}{\eps'}.
\]
Taking the minimum of the two bounds yields the claim.
\end{proof}
}

\subsection{Proofs for \cref{sec:support-stability}}

Next, we prove the Hausdorff stability bound for the optimal support.

\begin{proof}[Proof of \cref{thm:support-stability}]
We have \(\|h-h'\|_{L^\infty(\X\times\Y)} \le\overline C\,\Delta_*\) by \cref{cor:solution-map-lipschitz}, so that
\begin{equation}
\label{eq:slack-uniform-bound}
\|\sigma-\sigma'\|_{L^\infty(\X\times\Y)}
\le
\|h-h'\|_{L^\infty(\X\times\Y)}+\|c-c'\|_{L^\infty(\X\times\Y)}
\le
\delta_*.
\end{equation}

\smallskip
\noindent 
\emph{Step 1: Level set stability.} We first show, without requiring~\eqref{eq:exterior-nondegeneracy}, that
\begin{align}
\label{eq:support-sandwich}
\overline{\{z\in\Omega_{P\otimes Q}\cap\Omega_{P'\otimes Q'}:\sigma(z)>\delta_*\}}
&\subset
\Sigma'
\subset
\{z\in\Omega_{P'\otimes Q'}:\sigma(z)\ge -\delta_*\}
\end{align}
and
\begin{equation}
\label{eq:support-symmetric-difference}
\Sigma\triangle\Sigma'
\subset
\Big(
\{z\in\Omega_{P\otimes Q}\cap\Omega_{P'\otimes Q'}:|\sigma(z)|\le \delta_*,\ |\sigma'(z)|\le \delta_*\}
\Big)
\cup
(\Omega_{P\otimes Q}\triangle\Omega_{P'\otimes Q'}).
\end{equation}
Indeed, by \eqref{eq:optimal-support},
\[
\Sigma=\overline{\{z\in\Omega_{P\otimes Q}:\sigma(z)>0\}},
\qquad
\Sigma'=\overline{\{z\in\Omega_{P'\otimes Q'}:\sigma'(z)>0\}}.
\]
If \(z\in\Omega_{P\otimes Q}\cap\Omega_{P'\otimes Q'}\) satisfies \(\sigma(z)>\delta_*\), then \eqref{eq:slack-uniform-bound} yields
\[
\sigma'(z)\ge \sigma(z)-\delta_*>0,
\]
so \(z\in\Sigma'\). Passing to closures gives the first inclusion in \eqref{eq:support-sandwich}.
 Conversely, let \(z\in\Sigma'\). If \(\sigma(z)<-\delta_*\), then \eqref{eq:slack-uniform-bound} implies \(\sigma'(z)<0\), which by the continuity of \(\sigma'\) contradicts \(z\in\Sigma'\). Therefore \(\sigma(z)\ge -\delta_*\), which proves the second inclusion in \eqref{eq:support-sandwich}.

Now let \(z\in\Sigma\triangle\Sigma'\). If \(z\in\Omega_{P\otimes Q}\triangle\Omega_{P'\otimes Q'}\), then \(z\) already belongs to the right-hand side of \eqref{eq:support-symmetric-difference}. If \(z\notin\Omega_{P\otimes Q}\triangle\Omega_{P'\otimes Q'}\), then \(z\in\Omega_{P\otimes Q}\cap\Omega_{P'\otimes Q'}\). In that case, \(z\in\Sigma\triangle\Sigma'\) and a similar argument as above shows that \(|\sigma(z)|\le\delta_*\) and \(|\sigma'(z)|\le\delta_*\), proving \eqref{eq:support-symmetric-difference}.

\smallskip
\noindent
\emph{Step 2: Hausdorff stability.}
First note that
\begin{equation*}
\label{eq:product-support-hausdorff}
d_H(\Omega_{P\otimes Q},\Omega_{P'\otimes Q'})\le \Delta_\Omega.
\end{equation*}
Indeed, if \(z'=(x',y')\in\Omega_{P'\otimes Q'}\), we can choose
\(x\in\Omega_P\) and \(y\in\Omega_Q\) with
\(\|x-x'\|\le d_H(\Omega_P,\Omega_{P'})\) and
\(\|y-y'\|\le d_H(\Omega_Q,\Omega_{Q'})\). Then
\[
\|(x,y)-z'\|
\le
\Big(d_H(\Omega_P,\Omega_{P'})^2+d_H(\Omega_Q,\Omega_{Q'})^2\Big)^{1/2}
=
\Delta_\Omega,
\]
and the reverse inclusion is analogous.

Let \(z'=(x',y')\in\Sigma'\). Since \(z'\in\Omega_{P'\otimes Q'}\), we can choose \(z=(x,y)\in\Omega_{P\otimes Q}\) such that
\begin{equation}
\label{eq:nearest-product-support-point}
\|z-z'\|\le \Delta_\Omega.
\end{equation}
By the second inclusion in \eqref{eq:support-sandwich},
\[
\sigma(z')\ge -\delta_*.
\]
Moreover, since \(h=f\oplus g\) is \(\sqrt2L\)-Lipschitz and \(c\) is \(L\)-Lipschitz, \(\sigma=h-c\) is \((\sqrt2+1)L\)-Lipschitz. Hence
\[
\sigma(z)
\ge
\sigma(z')-(\sqrt2+1)L\|z-z'\|
\ge
-\delta_*-(\sqrt2+1)L\Delta_\Omega.
\]
If \(z\notin\Sigma\), then \eqref{eq:exterior-nondegeneracy} yields
\[
-a\,\dist(z,\Sigma)\ge \sigma(z)\ge -\delta_*-(\sqrt2+1)L\Delta_\Omega,
\]
that is,
\[
\dist(z,\Sigma)\le \frac{\delta_*+(\sqrt2+1)L\Delta_\Omega}{a}.
\]
If \(z\in\Sigma\), the same estimate is trivial. Combining this with \eqref{eq:nearest-product-support-point}, we obtain
\[
\dist(z',\Sigma)
\le
\|z'-z\|+\dist(z,\Sigma)
\le
\Delta_\Omega+\frac{\delta_*+(\sqrt2+1)L\Delta_\Omega}{a}.
\]
Therefore
\[
\sup_{z'\in\Sigma'}\dist(z',\Sigma)
\le
\left(1+\frac{(\sqrt2+1)L}{a}\right)\Delta_\Omega+\frac{\delta_*}{a}.
\]
The analogue holds when primed and unprimed quantities are exchanged, and~\eqref{eq:support-hausdorff-bound} follows.
\end{proof}

Finally, we verify the nondegeneracy condition under the concavity hypothesis, which holds in particular for the quadratic cost.

\begin{proof}[Proof of \cref{pr:concavity-exterior-nondegeneracy}]
If \(D_P=0\), then \(\Omega_P=\{x_0\}\). The second identity in \eqref{eq:potential-foc} gives \(\sigma(x_0,y)_+=\eps\) for all \(y\in\Omega_Q\), hence \(\sigma(x_0,y)=\eps\) for all \(y\in\Omega_Q\). By \eqref{eq:optimal-support}, \(\Omega_{P\otimes Q}\subset\Sigma\), so \(\Omega_{P\otimes Q}\setminus\Sigma=\emptyset\). We may therefore assume that \(D_P>0\).

Fix \(z=(x,y)\in\Omega_{P\otimes Q}\setminus\Sigma\). By the first-order condition \eqref{eq:potential-foc}, there exists \(x_*\in\Omega_P\) such that
\[
\sigma(x_*,y)\ge \eps.
\]

Define
\[
S_y:=\{\tilde{x}\in\Omega_P:\sigma(\tilde{x},y)\ge 0\}.
\]
Since \(x\mapsto \sigma(x,y)\) is continuous and concave on the convex set \(\Omega_P\), the set \(S_y\) is nonempty, closed, and convex. Moreover,
\begin{equation}
\label{eq:fiber-section-in-support}
S_y\times\{y\}\subset\Sigma.
\end{equation}
Indeed, if \(\tilde{x}\in S_y\) and \(\sigma(\tilde{x},y)>0\), then \((\tilde{x},y)\in\Sigma\) by \cref{eq:optimal-support}. If \(\sigma(\tilde{x},y)=0\), then for every \(t\in(0,1)\),
\[
\sigma\big((1-t)\tilde{x}+t x_*,y\big)
\ge
(1-t)\sigma(\tilde{x},y)+t\sigma(x_*,y)
\ge
t\eps>0,
\]
where we used concavity of $\sigma(\cdot,y)$ and convexity of \(\Omega_P\). Hence
\[
\big((1-t)\tilde{x}+t x_*,y\big)\in\Sigma
\qquad\text{for all }t\in(0,1),
\]
and letting \(t\downarrow0\) yields \((\tilde{x},y)\in\Sigma\), showing~\eqref{eq:fiber-section-in-support}.

Let $z=(x,y)\in \Omega_{P\otimes Q}\setminus\Sigma$. Then~\eqref{eq:fiber-section-in-support} implies that \(x\notin S_y\). By continuity of \(\tilde{x}\mapsto \sigma(\tilde{x},y)\) along the segment from \(x\) to \(x_*\), there exists \(\lambda\in(0,1)\) such that
\[
x_0:=(1-\lambda)x+\lambda x_* \in S_y
\qquad\text{and}\qquad
\sigma(x_0,y)=0.
\]
Using concavity once more, we obtain
\[
0=\sigma(x_0,y)
\ge
(1-\lambda)\sigma(x,y)+\lambda\sigma(x_*,y)
\ge
(1-\lambda)\sigma(x,y)+\lambda\eps,
\]
and therefore
\[
-\sigma(x,y)\ge \frac{\lambda}{1-\lambda}\eps\ge \lambda\eps.
\]
On the other hand,
\[
\dist(x,S_y)\le \|x-x_0\|=\lambda\|x-x_*\|\le \lambda D_P.
\]
Combining the last two displays yields
\[
-\sigma(x,y)\ge \frac{\eps}{D_P}\,\dist(x,S_y).
\]
Noting that $\dist((x,y),\Sigma)\le \dist(x,S_y)$ by \eqref{eq:fiber-section-in-support}, the claim~\eqref{eq:concavity-exterior-nondegeneracy} follows.

For the quadratic cost $c(x,y)=\frac12\|x-y\|^2$, it is known that $\sigma(x,y)$ is concave in each component, so that~\eqref{eq:fiberwise-concavity} holds. Specifically, $\sigma(x,y)=\langle x,y\rangle - \varphi(x)-\psi(y)$ where $\varphi,\psi$ are convex functions; see \cite{WieselXu.24} or \cite{GonzalezSanzNutz2024.Scalar}.
\end{proof}

\printbibliography

\end{document}